\documentclass[12pt,a4paper,twoside]{article}
\usepackage[T1]{fontenc}
\usepackage[english]{babel}
\usepackage{fancyhdr}
\usepackage{indentfirst} 
\usepackage{newlfont}
\usepackage{verbatim}
\usepackage{enumerate}
\usepackage{caption}
\usepackage{hyperref}
\usepackage{pdflscape}
\usepackage[all,cmtip]{xy}
\usepackage{amsthm}  
\usepackage{amssymb}  
\usepackage{amscd}
\usepackage{amsmath} 
\usepackage{latexsym} 
\usepackage{amsfonts}
\usepackage{amsbsy}   
\usepackage{mathrsfs} 
\usepackage[usenames,dvipsnames]{color}
\usepackage{graphicx}
\usepackage{color}

\theoremstyle{plain}
\newtheorem{teo}{Theorem}
\newtheorem{teoi}{Theorem}

\newtheorem{cor}{Corollary}

\theoremstyle{definition}

\newtheorem*{con}{Conjecture}

\theoremstyle{remark}
\newtheorem{oss}{Remark}
\newtheorem{ese}{Example}

\begin{document}

\selectlanguage{english}

\title{On the complexity of non-orientable Seifert fibre spaces
\footnotetext{A.~Cattabriga and M.~Mulazzani have been supported by the "National Group for Algebraic and Geometric Structures, and their Applications" (GNSAGA-INdAM) and University of Bologna, funds for selected research topics. S.~ Matveev has been supported by  RFBR grant  N17-01-0690. T.~Nasybullov has been supported by the Department of Mathematics of the University of Bologna, grant rep. N. 185/2015, and by the Research Foundation -- Flanders (FWO), app. 12G0317N.}}

\author{A.~Cattabriga - S.~Matveev - M.~Mulazzani - T.~Nasybullov}

\maketitle

\begin{abstract}
In this paper we deal with Seifert fibre spaces, which are  compact 3-manifolds admitting a foliation by  circles. We give a  combinatorial description for these manifolds in all the possible cases: orientable, non-orientable, closed, with boundary.   Moreover, we compute a potentially sharp upper bound for their complexity  in terms of the invariants of the combinatorial description, extending to the non-orientable case results by Fominykh and Wiest for the orientable case with boundary and by Martelli and Petronio for the closed orientable case.   Our upper bound is indeed sharp for all  Seifert fibre spaces contained  in the census of non-orientable closed 3-manifolds  classified with respect to complexity.   
\end{abstract}

\begin{section}{Introduction and preliminaries}
\label{preliminari}

The family of Seifert fibre spaces  (see \cite{Sc}) is  a  generalization of   Seifert's original one  (\cite{Se}), since it contains also manifolds locally modeled over solid Klein bottles (which are always non-orientable). This family coincides with the class of compact 3-manifolds foliated by  circles and  have a central role in Thurston geometrization theory (see for example \cite{Sc}). Indeed, in the closed case, each  Seifert fibre space is geometric and each   geometric 3-manifold  is  either a Seifert fiber space or admits  hyperbolic or  $\textup{Sol}$ geometry.  In other words the class of Seifert fibre spaces coincides with the class of geometric manifold admitting six of the eight possible geometries, that is $\mathbb E^3, \mathbb S^3, \mathbb S^2\times \mathbb R, \mathbb H^2\times \mathbb R, \textup{Nil}, \widetilde{SL_2(\mathbb R)}$. Moreover, Seifert fibre spaces with non-empty boundary are one of  the building blocks of the relevant class of  Waldhausen graph manifold (see \cite{Wa}).

  While the theory of Seifert fibre spaces is well established in the orientable case, including the construction of special spines and the   estimation of complexity, in the non-orientable one this is not the case: the knowledge  about  construction and classification  of  non-orientable Seifert fibre space, their special spines and complexity is very modest.  This  paper is devoted to  the closure of  this gap, dealing with  both closed and bordered case.

The notion of complexity for compact 3-dimensional manifolds  has been introduced by the second author in  \cite{M1} (see also \cite{M2}) as a way  to   measure how ``complicated''   a manifold is.  Indeed for closed irreducible and $\mathbb P^2$-irreducible manifolds, the  complexity coincides with the minimum number of tetrahedra needed to construct a manifold, with the only exceptions  of $S^3$, $\mathbb{RP}^3$ and $L(3,1)$, all having complexity zero.  
Moreover, complexity  is additive under connected sum, it does not increase when cutting along incompressible surfaces, and it is finite-to-one in the closed irreducible case.
The last property has been used in order to construct  a census of manifolds according to complexity:  exact values of it  are listed for the orientable  case at  \texttt{http://matlas.math.csu.ru/?page=search}  (up to complexity 12) and   for the non-orientable  case  at \texttt{https://regina-normal.github.io} (up to complexity 11). The main goal of the paper is to furnish a potentially sharp upper bound for the complexity of  Seifert fibre spaces, extending to the non-orientable case results of  \cite{FW} for the orientable case with boundary and of \cite{MP2} for the closed orientable case. It is worth noting that, in the non-orientable closed case, our upper-bound coincides with the exact value of the complexity for all tabulated manifolds (which are about 350).

The organization  of the paper is the following.  In Section \ref{sseifert} we recall the definition of Seifert fibre spaces and give a combinatorial description of them  by a set of parameters which completely classify the spaces, up to fibre-preserving homeomorphism,  proving the following result (see  Theorem \ref{combi} for more details). 
\begin{teoi} \label{combii} Every Seifert fibre space is uniquely  determined, up to fibre-preserving homeomorphism,  by the normalized set of parameters  
$$\left\{b;\left(\epsilon,g,(t,k)\right);\left(h_1,\ldots,h_{m_+}\mid k_1,\ldots, k_{m_{-}}\right);\left((p_1,q_1).\ldots,(p_r,q_r)\right)\right\}.$$

\end{teoi}

Section \ref{scomplexity} is devoted to the computation of the complexity.  We first   deal with the case with boundary obtaining the following upper bound (see Theorem \ref{comp_bordo_reg} for details), valid   both in the orientable and in the non-orientable case.

\begin{teoi} \label{comp_bordo_regi} Let $M=\left\{b;\left(\epsilon,g,(t,k)\right);\left(h_1,\ldots,h_{m_+}\mid k_1,\ldots, k_{m_{-}}\right);
\left((p_1,q_1),\ldots,(p_r,q_r)\right)\right\}$  be a  Seifert fibre space    with non-empty boundary. Then 
$$
 c(M)\leq t +\sum_{j=1}^r \max\left\{S(p_j,q_j)-3,0\right\}.
$$
\end{teoi} 

The second part of Section  \ref{scomplexity}   refers to the closed case:  we state and prove the result in complete generality, i.e., both for the orientable and for the non-orientable case (see  Theorem \ref{closed}). For   non-orientable manifolds, which is the relevant new case, we obtain the following result. 
\begin{teoi}\label{closedi} Let $M=\left\{b;\left(\epsilon,g,(t,k)\right);\left(\ \mid \ \right);
\left((p_1,q_1),\ldots,(p_r,q_r)\right)\right\}$  be a non-orientable closed irreducible and $\mathbb P^2$-irreducible Seifert fibre space, then 
$$
c(M)\leq 6(1-\chi)+6t+\sum_{j=1}^r \left(S(p_j,q_j)+1\right);
$$
where $\chi = 2-2g$ if the base space of $M$ is orientable and  $\chi=2-g$ otherwise.
\end{teoi}

\medskip

We end this section by recalling some preliminary notions on spines and complexity of 3-manifolds.

Let $S$ be a simplicial complex and let $\sigma^n, \delta^{n-1}\in  S$ be two open simplices such that (i) $\sigma^n$ is \textit{principal}, i.e. $\sigma^n$  is not a proper face of any simplex in $S$ and (ii) $\delta^{n-1}$ is a \textit{free} face of $\sigma^n$, that is $\delta^{n-1}$ is not a proper face of any  simplex in $S$ different from $\sigma^n$.   The transition from $S$ to $S\setminus (\sigma^n\cup\delta^{n-1})$ is called an \textit{elementary simplicial collapse}.  A polyhedron $P$ \textit{collapses} to a sub-polyhedron $Q$ (denoted by $P\searrow Q$) if for some triangulation $(S,L)$ of the pair $(P,Q)$ the complex $S$ collapses onto  $L$ by a finite sequence of elementary simplicial collapses.

A 2-dimensional polyhedron $P$ is said to be \textit{almost simple} if
the link of each point $x\in P$ can be embedded into $K_4$, the
complete graph with four vertices. In particular, the polyhedron is called \textit{simple} if the link is homeomorphic  to  either a circle, or  a circle with a diameter, or  $K_4$. A \textit{true vertex} of an
(almost) simple polyhedron $P$ is  a point $x\in P$ whose link is
homeomorphic to $K_4$. 

A \textit{spine} of a compact connected 3-manifold $M$ with $\partial M\ne \emptyset$
is a  polyhedron $P$ embedded in $\textup{int}(M)$ such that  $M$ collapses to $P$. A spine of a closed
connected 3-manifold $M$ is a spine of $M\setminus\textup{int}(B^3)$, where $B^3$ is an embedded closed 3-ball. 
If $P\subset \textup{int}(M)$ is a polyhedron,  then $P$ is a spine of $M$ if and only if $M\setminus P\cong\partial M\times [0,1)$, if $\partial M\ne \emptyset$, and 
$M\setminus P\cong B^3$ otherwise.  The \textit{complexity} $c(M)$ of $M$ is the minimum number of true
vertices among all almost simple spines of $M$.

To construct a  spine for a given manifold, we will  decompose the manifold  into blocks (also called bricks) by  cutting it along  embedded tori or Klein bottles, then providing  skeletons for each block and finally assembling  the pairs block-skeleton together. We adapt in this way  the definition of skeleton given in \cite{FW} in order to cover also the case of non-orientable Seifert fibre spaces  (a  general theory in this direction is developed  \cite{MP}).    Denote by $\mathcal{H}$ the class of  pairs $(M, \partial_-M)$, where  $M$ is a  compact connected 3-manifolds  whose (possibly empty) boundary is composed by tori and Klein bottles and $\partial_-M\subseteq\partial M$ is a (possibly empty) union of connected components of $\partial M$.   Moreover, let $\partial _+M=\partial M \setminus \partial_-M$. A \textit{skeleton of}  $(M, \partial_-M)\in\mathcal H$ is a  sub-polyhedron $P$ of $M$ such that (i) $P\cup\partial M$ is simple, (ii) $M\searrow (P\cup\partial_- M)$ if $\partial_+ M\ne\emptyset$ or  $(M\setminus\textup{int}(B^3))\searrow (P\cup\partial_- M)$ if $\partial_+M=\emptyset$, where $B^3$ is an embedded closed 3-ball, and (iii) for  each component $C$ of $\partial M$ the space $P\cap C$  is either empty or  a non-trivial theta curve\footnote{A  \textit{non-trivial} theta curve $\theta$ on a torus or a Klein bottle $C$ is  a subset of $C$  homeomorphic to the theta-graph (i.e., the graph with 2 vertices and 3 edges joining them),   such that $C\setminus \theta$ is an open disk.} or a non-trivial simple closed curve. Note that if  $P\cap \partial_+ M=\emptyset$ then a skeleton of $(M, \emptyset)$ is a simple spine for $M$. Given two pairs  $(M_1, \partial_-M_1)$ and $(M_2, \partial_-M_2)$ in $\mathcal H$, let $P_i$ be a skeleton of $(M_i, \partial_-M_i)$ for $i=1,2$. Take two components $C_1\subset \partial_+ M_1$ and $C_2\subset\partial_- M_2$ such that $P_i\cap C_i\ne\emptyset$ and $(C_1,P_1\cap C_1)$ is homeomorphic to     $(C_2,P_2\cap C_2)$ and fix a homeomorphism $\varphi:  (C_1,P_1\cap C_1)\to  (C_2,P_2\cap C_2)$. We  define a new pair $(W,\partial_-W)\in\mathcal H$, where $W=M_1\cup_{\varphi}M_2$ and $\partial_-W=\partial_ -M_1\cup_{\varphi} (\partial_-M_2\setminus C_2)$, and we say that  $(W,\partial_-W)$ is obtained by \textit{assembling}  $(M_1, \partial_-M_1)$ and $(M_2, \partial_-M_2)$ and the skeleton $P=P_1\cup_{\varphi} P_2$ of $(W,\partial_-W)$ is obtained  by \textit{assembling} $P_1$ and $P_2$.
\end{section}

\begin{section}{Seifert fibre spaces}
\label{sseifert}
In this section  we first recall the definition of Seifert fibre spaces  given in \cite{Sc}, then we give a combinatorial description of these spaces as well as a classification up to fibre-preserving homeomorphism, extending the results of \cite{Fi} to the case with boundary.

\begin{subsection}{Definitions and examples}
Denote by   $D=\{z\in\mathbb C\mid \vert z\vert\leq 1\}$ the closed unit disk  and by   $I=[0,1]$ the real unit interval. Moreover, let $S^1=\partial D$ and $D^+=\{z\in D\mid  \textup{Re}(z)\geq 0\}$.  Finally, let $N,K$ and $T$ be  the M\"obius strip, the Klein bottle and the torus,  respectively.

A \textit{fibred solid torus $T(p,r)$ of type $(p,r)$}  with $p,r\in\mathbb Z$, $p>0$ and  $\gcd(p,r)=1$,  is the 3-manifold  obtained from $D\times I$ by identifying $D\times \{0\}$ with $D\times\{1\}$ by the homeomorphism $\varphi_{p,r}$ defined by
$$\varphi_{p,r}:\ (z,0) \longmapsto (ze^{2i\pi\frac rp},1).$$
The fibred solid torus $T(p,r)$ is the union of the disjoint circles, called \textit{fibres},
$\bigcup_{k=0}^{p-1}\left\{z e^{2i\pi k\frac r p}\right\}\times I$ under the identification,  for each  $z\in D$. 
The fibre corresponding to $z=0$  is called \textit{the axis} of $T(p,r)$. 
The map obtained by collapsing each fibre to a point is a (regular) $S^1$-fibre bundle   if  $p=1$, while it   
has a singularity corresponding to the axis if $p>1$.  Moreover, we call  $T(1,0)$ \textit{the trivial solid torus} which  $p$-fold covers  $T(p,r)$.  It is well known that two fibred solid tori $T(p,r)$ and $T(p',r')$ are fibre-preserving homeomorphic  
if and only if  $p=p'$ and $r\equiv \pm r'$ $\mod p$. 

Analogously, we can define the \textit{(fibred) solid Klein bottle $SK$} as the manifold  which can be obtained from $D\times I$ 
by identifying $D\times \{0\}$ with $D\times \{1\}$ by  the (orientation reversing) homeomorphism $\varphi$ defined by
\footnote{Observe that the   replacing of $\varphi$ with another  reflection on $D$ does not affect the  fibre-preserving homeomorphism 
type of the resulting space.} 
$$\varphi:\ (z,0)\longmapsto(\overline{z},1).$$ 
The fibred solid Klein bottle is the union of the disjoint circles, called \textit{fibres},  
$(\left\{z\}\times I)\cup(\{\overline{z}\right\}\times I)$ under the identification,   for each $z\in D$. 
Note that $SK\cong N\times I$ and it is double covered by a trivial fibred solid torus.

Moreover, we call \textit{half solid torus} (resp.   \textit{half solid Klein bottle}) the fibred manifold obtained from  $D^+\times I$ by gluing $D^+\times \{0\}$ with  $D^+\times \{1\}$ by the restriction of  $\varphi_{1,0}$ (resp.  $\varphi$) to $D^+\times \{0\}$. 

A \textit{Seifert fibre space} $M$ is a compact connected 3-manifold  admitting a decomposition into disjoint circles, called \textit{fibres}, such that each fibre  has a neighborhood in $M$ which is  a union of fibres and it is   fibre-preserving homeomorphic to 
\begin{itemize} 
\item either  a fibred solid torus or  Klein bottle, if the fibre is contained in  $\textup{int}(M)$; 
\item either  a half  solid torus or a half  solid  Klein bottle, if the fibre is contained in $\partial M$.
\end{itemize}
Note that the original definition of Seifert manifolds given in \cite{Se} excludes the case of  fibred solid Klein bottles.  One of the interesting features of this more general definition is that the class of Seifert fibre spaces coincides with that of compact connected  3-manifolds foliated by circles (see \cite{E}).

We say that a fibre of $M$ is \textit{regular} if it has a neighborhood fibre-preserving homeomorphic  
to a trivial fibred solid torus or to a half solid torus,  and \textit{exceptional} otherwise.
Hence exceptional fibres are either  isolated, corresponding to the axis of $T(p,r)$ with $p>1$, or 
form properly embedded compact surfaces, corresponding to the points $\left(\{z\}\times I\right)/\sim_{\varphi}\subset SK$ with $\textup{Im}(z)=0$. Moreover, each connected exceptional surface is either a properly embedded annulus or it is a closed surface obtained by gluing together   the two boundaries  of an annulus,  so  it is either  a torus or a Klein bottle. We denote by $E(M)$ (resp. $SE(M)$) the union of all  isolated (resp. non-isolated) exceptional  fibres of $M$ and call $E$-\textit{fibre} (resp. $SE$-\textit{fibre}) any  fibre contained in $E(M)$ (resp. $SE(M)$). Finally, we set $SE(M)=CE(M)\cup AE(M)$, where $CE(M)$ contains  the  closed components of $SE(M)$, while $AE(M)$ contains the non-closed ones. Note that if $M$ is orientable then $SE(M)=\emptyset$.

The components of $\partial M$  are either tori or Klein bottles: the toric components are regularly fibred, while a  Klein bottle component is  either regularly fibred (see the left part of Figure~\ref{figklein})  or it contains two exceptional fibres of  $AE(M)$ (see the right part of Figure~\ref{figklein}).\\

\begin{figure}[h!]                      
\begin{center}                         
\includegraphics[width=10cm]{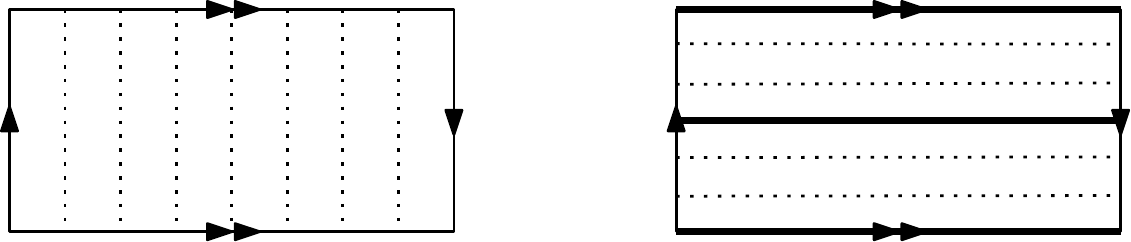}
\caption[legenda elenco figure]{The two different fibre structures of the Klein bottle boundary components of a Seifert fibre space.}\label{figklein}
\end{center}
\end{figure}

Given  a Seifert fibre space $M$, denote by $B$ the space obtained by collapsing  each fibre to a point and by  $f: M\to B$  the projection map. If $P\in B$ is the projection of a regular fibre $\phi$, then a tubular neighborhood of $P$  is either a disk (if $\phi\subset\textup{int}(M)$) or a half-disk (if $\phi\subset\partial M$). The possible models around points which are projections of an exceptional fibre are depicted in  Figure \ref{modelli}. As a consequence, $B$ is  a compact 2-dimensional orbifold,  called \textit{base space},    whose singular locus $\mathcal S$ coincides with  the projection of all   the exceptional fibres $E(M)\cup SE(M)$ of $M$  (thick lines and points in the figures represent singularities of the orbifold).   More precisely, the singularities of the orbifold are  (see also \cite{Sc}):
\begin{itemize}
\item cone points of cone angle $2\pi/p$ for $p>1$, corresponding to $E$-fibres having a neighborhood  fibre-preserving homeomorphic to
$T(p,r)$;
\item reflector arcs,  corresponding to  components of $AE(M)$ (i.e., annulus exceptional surfaces);
\item reflector circles, corresponding to  components of $CE(M)$ (i.e., tori or Klein bottles exceptional surfaces). 
\end{itemize}
Note that the orbifold $B$ has no corner points in its singular locus  and that  the restriction of $f$ to the counter-image  of the complement  of an open  tubular neighborhood  $N(\mathcal S)$ of  $\mathcal S\subset B$  is an $S^1$-bundle over the compact surface $B\setminus N(\mathcal S)$.

\begin{figure}[h!]                      
\begin{center}                         
\includegraphics[width=9cm]{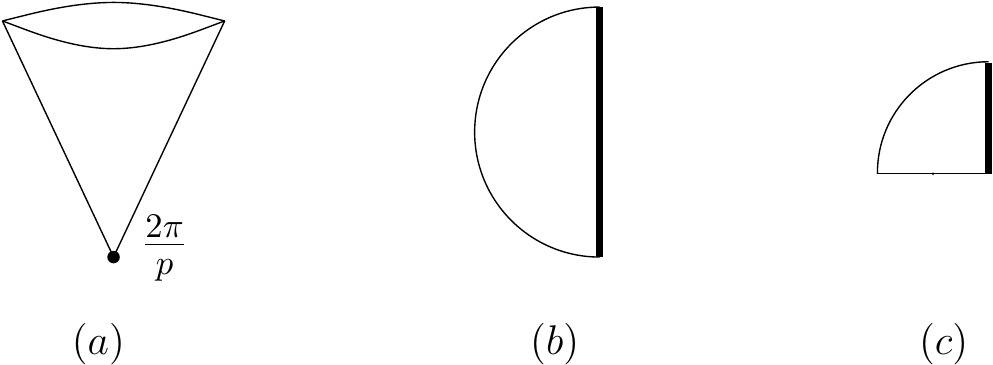}
\caption[legenda elenco figure]{Local models for singular points of the  base orbifold $B$ of a Seifert fibre space:  (a) cone points, (b) reflector points corresponding to internal fibres  and  (c) reflector points corresponding to boundary fibres.}\label{modelli}
\end{center}
\end{figure}

\begin{ese} \label{es_seifert} The  solid torus $D^2\times S^1$ and the solid
Klein bottle $SK$  are both examples of Seifert fibre spaces with non-empty boundary: the first one admits infinitely many fibre space structures $T(p,r)$ with one isolated exceptional fibre when $p>1$, while  $SK$  admits a unique Seifert fibre space structure, having  an  annulus  as exceptional set (see  Figure~\ref{modelli} (a) and (b) for a representation of  the base orbifold).  Other interesting examples of Seifert fibre spaces are 
the two $N$-bundles over $S^1$, namely $N\times S^1$ and  $N\widetilde{\times} S^1$. We recall that $N\widetilde{\times} S^1$ is the manifold obtained  from $N\times I$ by gluing  $(x,0)$ with $(g(x),1)$, where, referring to Figure~\ref{figanms}, the map $g$ is the composition of a  reflection along the exceptional fiber of $N$ (the thick line) with a reflection along the $\ell$ axis. In this case $\partial(N\widetilde{\times} S^1)=K$. The manifold $N\times S^1$  admits the trivial product fibration (without exceptional fibres) and a Seifert fibration  with a toric exceptional 
surface; while $N\widetilde{\times} S^1$  admits a Seifert fibration having an  isolated exceptional fibre of type $(2,1)$ and an exceptional  annulus,   and another one with a Klein bottle exceptional surface. The pictures in the first two rows of  Figure~\ref{example}  represent the base orbifold of all such fibrations (the meaning of the  labels in the figure  will  be explained in the next subsection). 
\end{ese} 

\begin{figure}[h!]                      
\begin{center}                         
\includegraphics[width=5cm]{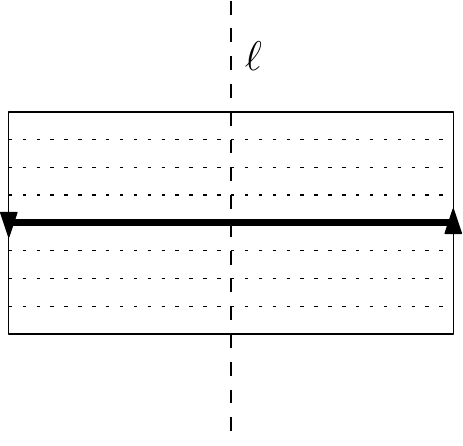}
\caption[legenda elenco figure]{The M\"obius  strip  foliated by circles.}\label{figanms}
\end{center}
\end{figure}
\end{subsection}

\begin{subsection}{Combinatorial description and fibre-preserving classification}
\label{construction}

A combinatorial description for  closed Seifert fibre spaces   is given in \cite{Fi} as well as the classification of these spaces up to fibre-preserving homeomorphisms. In this section  we  extend that description to the   case with boundary. 
  
Let $M$ be a Seifert fibre space with non-empty boundary and without exceptional fibres, then $f:M\to B$ is an $S^1$-bundle over $B$.
Denote by $\omega:H_1(B)\to \{1,-1\}$ the group homomorphism such that $\omega(\alpha)=1$ if and only if the 
orientation of a fibre in $M$ is preserved when a representative  loop of $\alpha$ in $B$ is traversed.  
If $B$ has genus $g\geq 0$ and $n>0$ boundary components then, referring to Figure~\ref{sup_bordo},
$$H_1(B)=\langle a_i,b_i,s_j\mid s_1+\cdots +s_{n}=0\rangle_{i=1,\ldots,g,\ j=1,\ldots,n}$$
if $B$ is orientable, and  
$$H_1(B)=\langle v_i,s_j\mid s_1+\cdots +s_{n}+ 2v_1+\cdots +2v_g=0\rangle_{i=1,\ldots,g,\ j=1,\ldots,n}\qquad (g\geq 1)$$
if $B$ is non-orientable.  We say that the $S^1$-bundle $f:M\to B$ is of type:

- $o_1$   if $\omega(a_i)=\omega(b_i)=1$ for all $i=1,\ldots,g$;

- $o_2$   if $\omega(a_i)=\omega(b_i)=-1$ for all $i=1,\ldots,g$ ($g\geq 1$);

- $n_1$   if $\omega(v_i)=1$ for all $i=1,\ldots,g$ ($g\geq 1$);

- $n_2$   if $\omega(v_i)=-1$ for all $i=1,\ldots,g$ ($g\geq 1$);

- $n_3$   if $\omega(v_1)=1$ and $\omega(v_i)=-1$ for all $i=2,\ldots,g$ ($g\geq 2$);

- $n_4$   if $\omega(v_1)=\omega(v_2)=1$ and $\omega(v_i)=-1$ for all $i=3,\ldots,g$ ($g\geq 3$).

\begin{figure}[h!]                      
\begin{center}                         
\includegraphics[width=13cm]{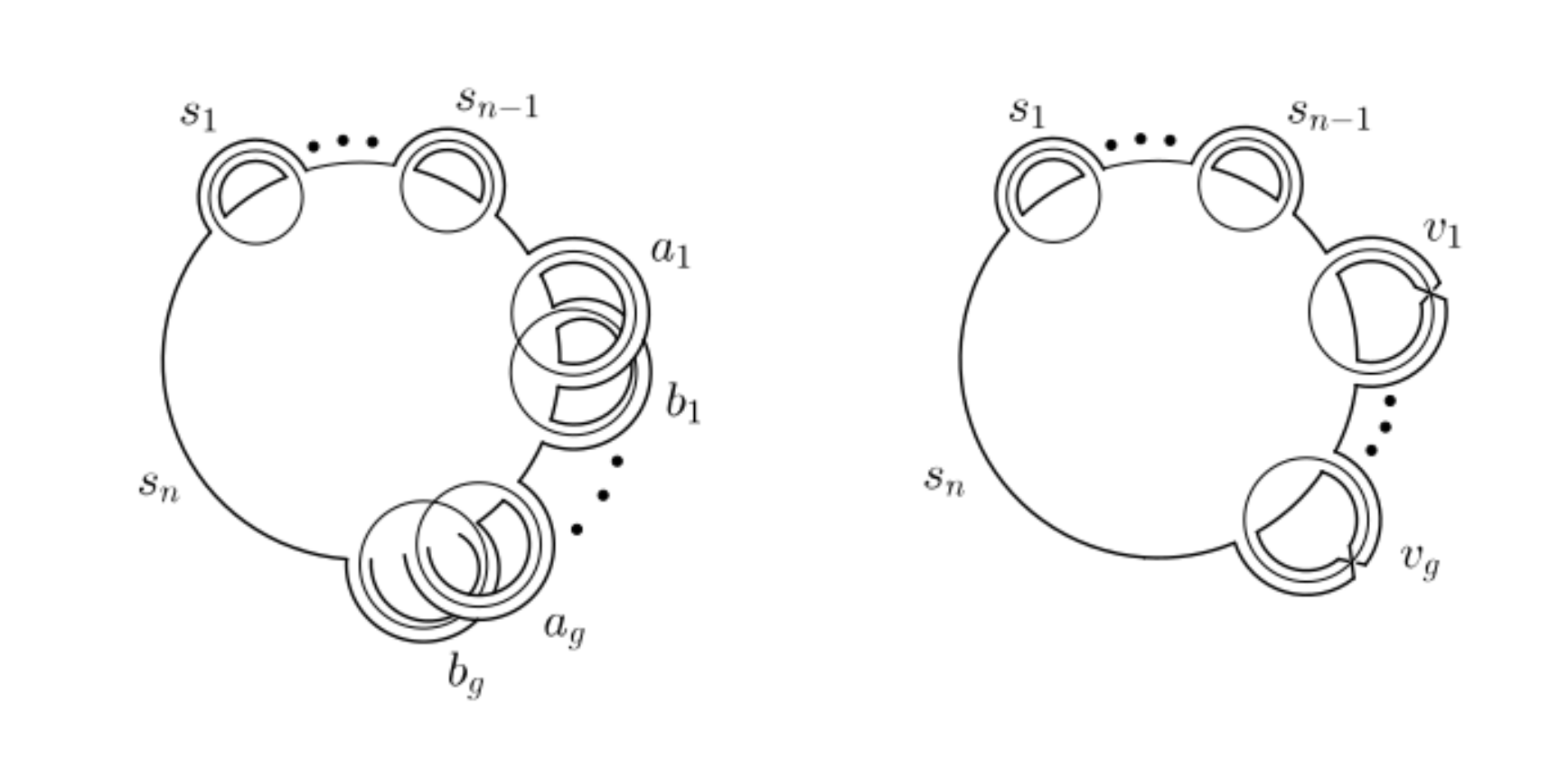}
\caption[legenda elenco figure]{Generators of $H_1(B)$.}\label{sup_bordo}
\end{center}
\end{figure}

The following theorem describes the classification of  $S^1$-bundles over a fixed surface, up to fibre-preserving homeomorphisms.

\begin{teo}[\upshape\cite{Fi}] \label{bundle} Let $B$ be a compact connected surface with non-empty boundary.
The fibre-preserving homeomorphism classes of $S^1$-bundles over $B$ are in 1-1 correspondence with the pairs $(k,\epsilon)$,
where $k$ is an even non-negative number  which counts the number of $s_j$ such that $\omega(s_j)=-1$ and 
(i)   $\epsilon=o_1,o_2$ when $B$ is orientable and  $\epsilon=n_1,n_2,n_3,n_4$ when $B$ is non-orientable, if $k=0$ or (ii)
 $\epsilon=o$ with $o:=o_1=o_2$  when  $B$ is orientable and $\epsilon=n $ with $n:=n_1=n_2=n_3=n_4$  when $B$ is non-orientable, if  $k>0$.
\end{teo}

Now we are ready to introduce the combinatorial description for Seifert fibre spaces. Let
\begin{itemize}
\item $g,t,k,m_+,m_-,r$ be non-negative integers such that $k+m_-$  is even and $k\leq t$;    
\item $\epsilon$ be a symbol belonging to the set $\cal E=\{o,o_1,o_2,n,n_1,n_2,n_3,n_4\}$  such that (i) $\epsilon =o,n$ if and only if $k+m_-> 0$, (ii) if $\epsilon=n_4$ then $g\geq 3$, (iii) if $\epsilon=n_3$ then $g\geq 2$ and (iv) if  $\epsilon=o_2,n, n_1,n_2$ then $g\geq 1$; 
\item  $h_1,\ldots,h_{m_+}$ and  $k_1,\ldots, k_{m_-}$ be non-negative integers such that $h_1\leq \cdots\leq h_{m_+}$ and  $k_1\leq \cdots\leq k_{m_-}$;
\item $(p_j,q_j)$ be lexicographically ordered  pairs of coprime integers such that $0<q_j<p_j$ if $\epsilon=o_1,n_2$ and $0<q_j\leq p_j/2$ otherwise, for $j=1,\ldots, r$;  
\item $b$ be an arbitrary integer if $t=m_+=m_-=0$ and $\epsilon=o_1,n_2$; $b=0$ or $1$ if $t=m_+=m_-=0$ and  $\epsilon=o_2,n_1,n_3,n_4$ and no $p_j=2$; $b=0$ otherwise.
\end{itemize}

The previous parameters with the given conditions are called {\it normalized}, and we denote by $$\left\{b;\left(\epsilon,g,(t,k)\right);\left(h_1,\ldots,h_{m_+}\mid k_1,\ldots, k_{m_{-}}\right);\left((p_1,q_1),\ldots,(p_r,q_r)\right)\right\}$$ 
the Seifert fibre space constructed as follows.  

If $b=0$, denote by $B^*$  a compact connected genus $g$ surface having  $s=r+t+m_++m_-$ boundary components and being orientable 
if $\epsilon=o,o_1,o_2$ and non-orientable  otherwise. By Theorem~\ref{bundle} there is   a unique   $S^1$-bundle over $B^*$ associated to the pair 
$(k+m_-,\epsilon)$, up to fibre-preserving homeomorphism: call it $M^*$ (see Remark~\ref{costruzione}  for the details of this construction). 
Note that $M^*$ has $k+m_-$ boundary components which are Klein bottles and the remaining  $r+t-k+m_+$ ones are tori.
Denote by $c_1,\ldots, c_{s}$  the boundary components of $B^*$, numbering them  such that the last $k+m_-$ correspond to
Klein bottles in $M^*$.
Let   $\partial_1 B^*=c_{1}\cup\ldots\cup c_{r+t-k+m_+}$ and  $\partial_2 B^*= \partial B^* \setminus \partial_1 B^*$. 
Finally, denote by   $s^*:B^*\to M^*$  a section of $f^*:M^*\to B^*$. 
\begin{itemize}

\item[(i)] For  $j=1,\ldots,r$ fill the  toric boundary component $(f^*)^{-1}(c_{j})$ of $M^*$ with a solid torus by sending the boundary  of a meridian disk of the solid torus into the curve $p_jd_j+q_jf_j$,  where $f_j$ is a fibre and $d_j=s^*(c_{j})$; 

\item[(ii)]  for  $i=1,\ldots,m_+$ (resp. $j=1,\ldots,m_-$)    
consider $h_i$ (resp. $k_j$)  disjoint closed  arcs  
inside the boundary component  $c_{i+r}$ of $\partial_1 B^*$ (resp. $c_{j+r+t-k+m_+}$ of $\partial_2 B^*$) and,  for each arc and each point $x$ of  the arc, 
attach a   M\"obius strip   along the boundary to  the fibre $(f^*)^{-1}(x)$, where  the M\"obius strip is foliated by circles as depicted in Figure~\ref{figanms}.  On the whole, we attach $h_i$ (resp. $k_j$)
disjoint copies of $N\times I$ to the boundary 
of $M^*$ corresponding to the counter-image of $c_{i+r}$ (resp. $c_{j+r+t-k+m_+})$. So the boundary component  remains unchanged if $h_i=0$ (resp. $k_j=0$) and it  is partially filled 
 otherwise. In the latter case instead of the initial boundary component we have 
$h_i$ (resp. $k_j$) Klein bottle boundary components;

\item[(iii)]  for  $i=1,\ldots,t-k$ (resp. $j=1,\ldots,k$) glue a copy of $N\times S^1$ (resp.  $N\widetilde{\times} S^1$) to each toric (resp. Klein bottle) boundary component  of   $M^*$  along the boundary via a homeomorphism which is fibre-preserving with respect to the fibration depicted in Figure \ref{example} ($a'$) (resp. $(b')$). Namely,  as in the previous step, for each point $x\in c_{i+r+m_+}$  (resp. $x\in c_{j+r+t-k+m_++m_-}$) we attach a M\"obius strip  along the boundary to  the fibre $(f^*)^{-1}(x)$.
\end{itemize}

If $b\neq 0$ (and therefore $t=m_+=m_-=0$) we modify the above construction as follows: we take a surface  $B^*$ with  $r+1$ boundary components and fill  the first $r$-ones boundary components of $M^*$ as described in (i)  and the last one by sending    the boundary  of a meridian disk of the solid torus into $d_{r+1}+bf_{r+1}$ (i.e., with  $(p_{r+1},q_{r+1})=(1,b)$).

The resulting manifold is the Seifert fibre space $$M=\left\{b;\left(\epsilon,g,(t,k)\right);\left(h_1,\ldots,h_{m_+}\mid k_1,\ldots, k_{m_{-}}\right);\left((p_1,q_1),\ldots,(p_r,q_r)\right)\right\}.$$  

Note that when $t=m_+=m_-=0$, the above construction gives the classical closed Seifert fibre space $\left(b;\epsilon,g;(p_1,q_1),\ldots,(p_r,q_r)\right)$ of \cite{Se}.

From the above construction it follows that the exceptional set of $M$ is composed by:  
 (i) an $E$-fibre  of type\footnote{Note  that a fibred tubular neighborhood of an $E$-fibre of type $(p_j,q_j)$ is fibre-preserving  equivalent to $T(p_j,r_j)$ with  $r_jq_j\equiv 1\mod p_j$.}    $(p_j, q_j)$ for $j=1,\ldots, r$, (ii)    $t$  closed exceptional  surfaces, $k$ of which are Klein bottles while the remaining $t-k$ are tori and (iii) $t'=h_1+\cdots+h_{m_+}+k_1+\cdots+k_{m_{-}}$ exceptional  surfaces homeomorphic to  annuli. Moreover, the boundary of $M$ has $t'$ components  which are Klein bottles with two exceptional fibres (contained in $AE(M)$) and, for each $h_i=0$ (resp.   $k_j=0$),  a toric (resp. Klein bottle) boundary component without exceptional fibres.

The singular locus $\mathcal S$ of the base orbifold $B$   (that will be depicted using thick lines and  points in figures)  consists of: (i)
$r$ cone points of cone angles $2\pi/p_1,\ldots,2\pi/p_r$ (in figures each cone point will be decorated with the corresponding pair $(p_j,q_j)$), (ii) $t$ reflector circles and (iii) $t'$ reflector arcs.
The underlying surface of the orbifold has genus $g$ and it is orientable if and only if $\epsilon=o,o_1,o_2$.   
Moreover,  it has $m_++m_-+t$ boundary components: one boundary component containing
$h_i$ (resp. $k_j$) disjoint reflector arcs for each $i=1,\ldots,m_+$ (resp. $j=1,\ldots,m_-$), and one boundary components for each reflector circle. We  decorate by the  symbol ``$-$'' each  boundary component of the underlying surface having a Klein bottle as counterimage in $M$.

\begin{oss}
Conditions on the invariants ensuring the orientability and the closeness of a Seifert fibre space 
are the following:
\begin{itemize}
\item[(i)] $M$ is orientable if and only if $t=m_-=0$,   $h_i=0$ for all $i=1,\ldots, m_+$ and $\epsilon=o_1,n_2$;
\item[(ii)] $M$ is closed if and only if  $m_+=m_-=0$. In this case  the combinatorial description coincide with the one of \cite{Fi}.
\end{itemize}
\end{oss}

\begin{ese} 
The   Seifert fibre space   $\left\{0;\left(o,4,(1,1)\right); \left(1\mid 0\right);\left((3,1),(5,2)\right)\right\}$ is depicted in  Figure~\ref{figcomb}. It has two $E$-fibres of type $(3,1)$ and $(5,2)$, one Klein bottle exceptional surface and one annulus exceptional surface. The boundary consists of two Klein bottles, one with two exceptional fibres and  another without exceptional fibres.
\end{ese} 
 
\begin{figure}[h!]                      
\begin{center}                         
\includegraphics[width=9cm]{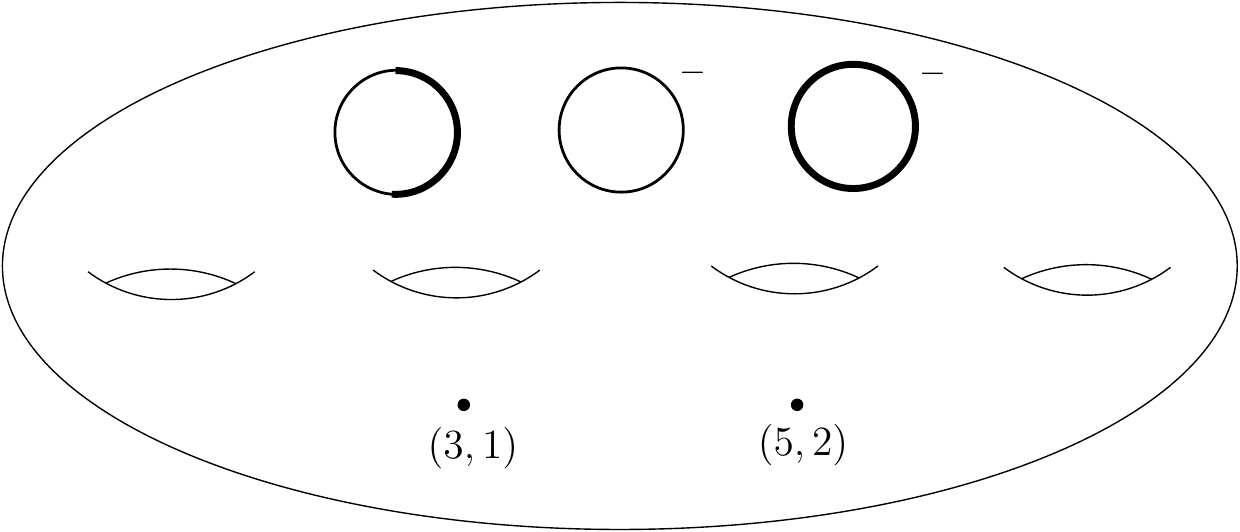}
\caption[legenda elenco figure]{The Seifert fibre space $\left\{0;\left(o,4,(1,1)\right); \left(1\mid 0\right); \left((3,1),(5,2)\right)\right\}$.}\label{figcomb}
\end{center}
\end{figure}

\begin{oss}
\label{rem}
Let  $M$ be a Seifert fibre space such that  $M\setminus SE(M)$ is orientable, and therefore  $M=\left\{b;\left(\epsilon,g,(t,0)\right);\left(h_1,\ldots,h_{m_+}\mid \ \right);\left((p_1,q_1),\ldots,(p_r,q_r)\right)\right\}$  with  $\epsilon\in \{o_1,n_2\}$. If $M$ is closed and orientable   (i.e., $t=m_+=0$), by reversing a fixed  orientation on $M$  we obtain $$-M=\left\{-b-r;\left(\epsilon,g,(0,0)\right);\left(\ \mid\ \right);\left((p_1,p_1-q_1),\ldots,(p_r,p_r-q_r)\right)\right\}$$ (see \cite[p.18]{O}). So, if we do not take care of the orientation,  we can suppose  $b\geq -r/2$. Moreover, if $b=-r/2$ we can assume $0<q_{l}< p_{l}/2$, where $l$ is the minimum $j$, if any, such that $p_{j}>2$.
In all the other cases (i.e., $\partial M\ne \emptyset$ or $M$ non orientable)  $b=0$, and, reversing the orientation on $M\setminus SE$ we get the equivalent space $\left\{0;\left(\epsilon,g,(t,0)\right);\left(h_1,\ldots,h_{m_+} \mid\ \right);\left((p_1,p_1-q_1),\ldots, (p_r,p_r-q_r)\right)\right\}$. So, we can suppose $0<q_{l}<p_{l}/2$, where $l$ is as above.\\

\end{oss}

\begin{teo} \label{combi} Every Seifert fibre space is uniquely  determined, up to fibre-preserving homeomorphism,  by the normalized set of parameters  
$$\left\{b;\left(\epsilon,g,(t,k)\right);\left(h_1,\ldots,h_{m_+}\mid k_1,\ldots, k_{m_{-}}\right);\left((p_1,q_1),\ldots,(p_r,q_r)\right)\right\},$$
and, when  $M\setminus SE(M)$ is orientable (i.e., $\epsilon\in\{o_1,n_2\}$),  by the following additional conditions: (i) if $M$ is closed and orientable, then   $b\geq - r/2$ and,  if $b= -r/2$,  $0<q_{l}< p_{l}/2$,  (ii) if  $M$ is non-closed or non-orientable then $0<q_{l}< p_{l}/2$; where  $l$ is the minimum $j$, if any,  such that $p_{j}>2$.
\end{teo}
\begin{proof}
If $M=\left\{b;\left(\epsilon,g,(t,k)\right);\left(h_1,\ldots,h_{m_+}\mid k_1,\ldots, k_{m_{-}}\right);\left((p_1,q_1),\ldots,(p_r,q_r)\right)\right\}$ 
and $\bar M=\left\{\bar b;\left(\bar \epsilon,\bar g,(\bar t,\bar k)\right);\left(\bar h_1,\ldots,\bar h_{\bar m_+}\mid \bar k_1,
\ldots, \bar k_{\bar m_{-}}\right);\left((\bar p_1,\bar q_1),\ldots,(\bar p_{\bar r},\bar q_{\bar r})\right)\right\}$ 
are fibre-preserving homeomorphic then, by looking at the boundaries of the base orbifolds, it is clear that  $m_+=\bar m_+$, $m_-=\bar m_-$,   $h_{i}=\bar h_i$, $k_{j}=\bar k_j$ for $i=1,\ldots,m_+$ and $j=1,\ldots,m_-$. If  we fill, respecting the fibration, each boundary component with two exceptional fibres with a solid Klein bottle, each toric boundary component with  $N\times S^1$, and each Klein bottle boundary component without exceptional fibres with  $N\widetilde \times S^1$,
we obtain the two closed Seifert fibre spaces 
$\left\{b;\left(\epsilon,g,(t,k)\right);\left(\ \mid \ \right);\left((p_1,q_1),\ldots,(p_r,q_r)\right)\right\}$ and $\left\{\bar b;\left(\bar \epsilon,\bar g,(\bar t,\bar k)\right);\left(\ \mid \ \right);\left((\bar p_1,\bar q_1),\ldots,(\bar p_{\bar r},\bar q_{\bar r})\right)\right\}$. So the result follows directly from \cite[Theorem 2]{Fi} and Remark \ref{rem}.
\end{proof}

From now on  we will always suppose the parameters of Seifert fibre spaces to be normalized.

\begin{figure}[h!]                      
\begin{center}                         
\includegraphics[width=8cm]{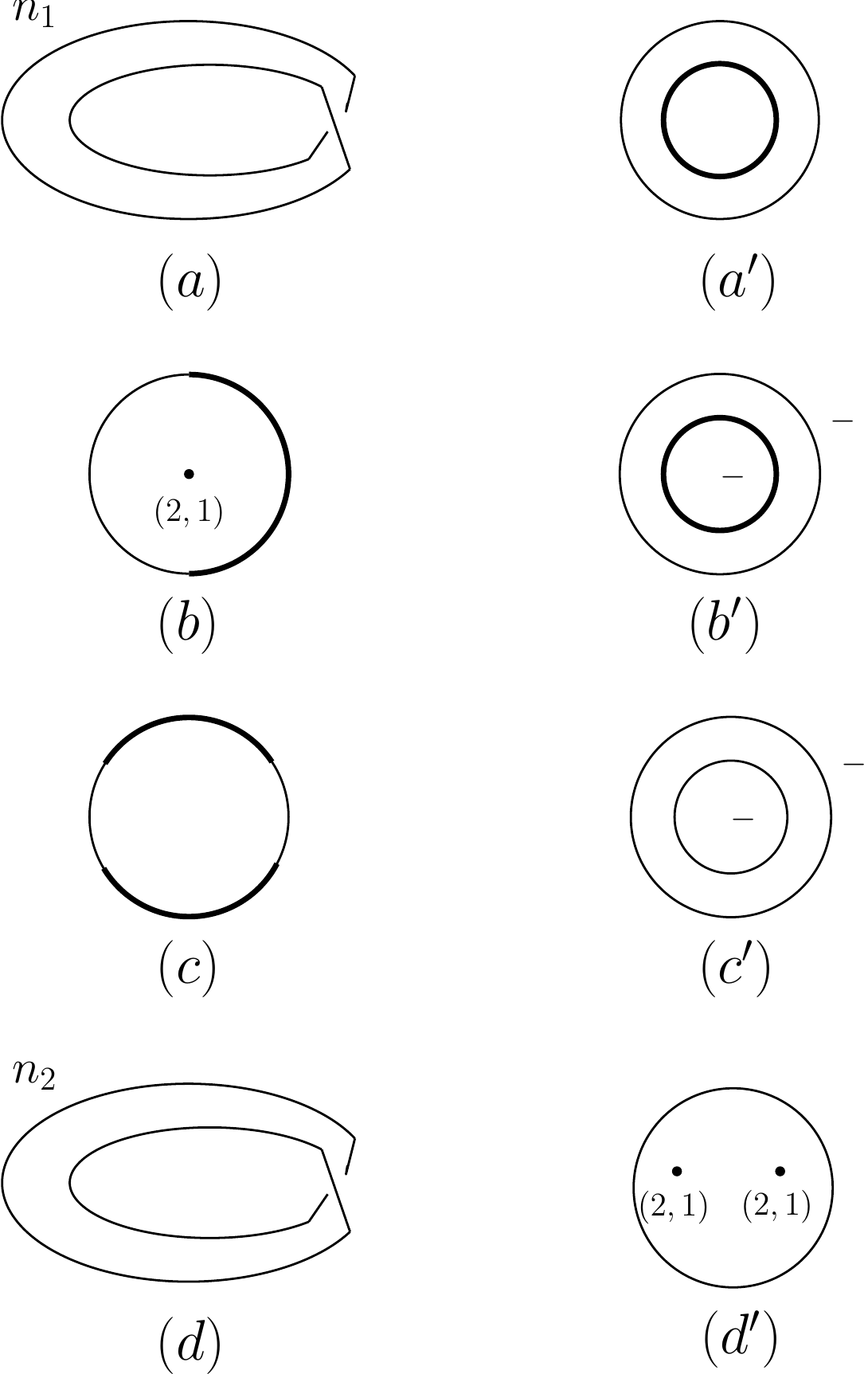}
\caption[legenda elenco figure]{The Seifert fibre structures over:   $(a,\ a')$ $N\times S^1$, $(b,\ b')$ $N\widetilde{\times} S^1$, $(c,\ c')$ $K\times I$, $(d,\ d')$ $K\widetilde{\times}I$.}\label{example}
\end{center}
\end{figure}

\begin{ese}  \label{es_seifert_combi} 
The solid torus $D^2\times S^1$ admits the combinatorial descriptions
 $\left\{0;\left(o_1,0,(0,0)\right); \left(0\mid\ \right); (p,q)\right\}$ if $p>1$, and $\left\{0;\left(o_1,0,(0,0)\right); \left(0\mid\ \right); \ \right\}$ if \hbox{$p=1$}, while  the solid Klein bottle admits the description
 $\left\{0;\left(o_1,0,(0,0)\right); \left(1\mid\ \right); \ \right\}$ (see also Example \ref{es_seifert}). Other important  examples are depicted in Figure~\ref{example}:  the manifold $N\times S^1$ has two different Seifert  space structures, up to fibre-preserving homeomorphism, namely $(a)$ the trivial $S^1$-bundle over $N$, whose description is  $\left\{0;\left(n_1,1,(0,0)\right); \left(0\mid\ \right); \ \right\}$, and  $(a')$ that is $\left\{0;\left(o_1,0,(1,0)\right); \left(0 \mid\ \right); \ \right\}$. Also  $N\widetilde{\times} S^1$  can be fibred both as   $\left\{0;\left(o_1,0,(0,0)\right); \left(1\mid\ \right); (2,1) \right\}$, depicted in $(b)$,  and  $\left\{0;\left(o,0,(1,1)\right); \left(\ \mid 0\right); \ \right\}$, depicted in $(b')$.  Pictures   $(c)$  and $(c')$ represent the two possible Seifert structures over  $K\times I$ (i.e., $\left\{0;\left(o_1,0,(0,0)\right); \left(2\mid\ \right); \  \right\}$ and   $\left\{0;\left(o,0,(0,0)\right); \left(\ \mid 0,0\right);\  \right\}$, respectively). Finally,  
in $(d)$  and  $(d')$ there are two different    Seifert structures of $K\widetilde{\times}I$ a twisted  $I$-bundle over $K$, that are $\left\{0;\left(n_2,1,(0,0)\right); \left(0 \mid \ \right); \ \right\}$ and   $\left\{0;\left(o_1,0,(0,0)\right); \left(0 \mid\ \right); ((2,1),(2,1))\right\}$, respectively.  As proved in \cite[Proposition A.1]{AM}, the previous four manifolds  and  $T\times I=\left\{0;\left(o_1,0,(0,0)\right); \left(0,0\mid\ \right);\ \right\}$ are all the possible $I$-bundles over the torus  ($T\times I$ and $T\widetilde{\times}I=N\times S^1$) and the Klein bottle ($K\times I$, $K\widetilde{\times}I$ and $K\widetilde{\widetilde{\times}}I=N\widetilde{\times}S^1$).
\end{ese}

\begin{oss} \label{costruzione} We  recall how to construct an $S^1$-bundle of type $(k,\epsilon)$ over the compact connected surface $B^*$ 
with $\partial B^*\ne\emptyset$. The surface  $B^*$  is homeomorphic to a disk 
with $r$ attached bands $\beta_1,\ldots, \beta_r$,  where $r=\textup{rank}(H_1(B^*))$. Let
 $y_1,\ldots,y_r$ be the generators  of $H_1(B^*)$   depicted in  Figure~\ref{sup_bordo} (i.e., $y_l=a_i,b_i,s_j$ if $B^*$ is orientable and $y_l=v_i,s_j$ otherwise). For $l=1,\ldots, r$ denote by $d_l$ the oriented cocore of  $\beta_l$; cutting $B^*$ along $A=d_1\cup\cdots\cup d_l$ we obtain  a disk $\Delta$.    Let   $d'_l$ and $d''_l$  be the two oriented copies of $d_l$ in $\Delta$ and, for each $x\in d_l$, denote by $x'$ and $x''$ the two  points in $d'_l$ and $d''_l$ corresponding to $x$, respectively. Finally, let $\omega: H_1(B^*)\to \{-1,1\}$ be the homomorphism  associated to the pair $(k,\epsilon)$. Since $\Delta$ is contractible, $\Delta\times S^1$ is the unique  $S^1$-bundle over $\Delta$. Attach $d'_l\times S^1$ to $d''_l\times S^1$ via $(x',e^{i\theta})\mapsto(x'',e^{i\theta})$  if $\omega(x_l)=1$ and via $(x',e^{i\theta})\mapsto(x'',e^{-i\theta})$  otherwise. The resulting manifold $M^*$ is the required  $S^1$-bundle over $B^*$. 
\end{oss}

\end{subsection}

\end{section}

\begin{section}{Complexity of Seifert fibre spaces}
\label{scomplexity}
\begin{subsection}{The case with non-empty boundary}
In this subsection we analyze the case $\partial M \ne \emptyset$ describing a (almost) simple spine 
for $M$ and using it to  give an upper bound for the complexity of the manifold.

In \cite{FW} the authors construct a (almost) simple spine   for  orientable Seifert fibre spaces,  and therefore with $SE(M)=\emptyset$, 
obtaining an upper bound for the complexity. Let us recall their result.
For two coprime integers $p,q$ with $0<q<p$  denote by $S(p,q)$ 
the sum of the coefficients of the expansion of $p/q$ as a continued fraction: 
 
 $$ \textup{if}\quad \frac{p}{q}= a_1+\cfrac{1}{\ddots +\cfrac{1}{a_{k-1} +\cfrac{1}{a_k}}},\qquad \textup{then}
 \ S(p,q)=a_1+\cdots+a_k.$$
 
\begin{teo}[\upshape\cite{FW}] \label{eugeni_or}  Let $M$ be an orientable Seifert fibre space with non-empty boundary, having  E-fibres of types $(p_1,q_1),\ldots, (p_r,q_r)$ with
 $p_j>q_j>0$ for all $j=1,\ldots,r$. Then 
 $$c(M)\leq \sum_{j=1}^r \max\left\{S(p_j,q_j)-3,0\right\}.$$ 
\end{teo}
This  theorem is proved  by finding  a spine of $M$: such a spine  is obtained  decomposing   $M$ into blocks  and then assembling  the  skeletons of the different blocks together.  In order to generalize the result to the case  $SE(M)\ne\emptyset$ we  adapt one of the blocks, the ``main'' one, in order to include $AE(M)$, and describe a new type of  block for the components of  $CE(M)$.\\

\textbf{Main block.} Let $M_0$  be a Seifert fibre space such that $\partial M_0\ne \emptyset$ and  $E(M_0)=CE(M_0)=\emptyset$ and 
let $f_0:M_0\to B_0$ be the projection map. Moreover, suppose that if $B_0$ is a disk, then the boundary of $M_0$ has at least two components (so there are at least two reflector arcs in the boundary of the disk). We take a decomposition of $\partial M_0$ into $\partial_+ M_0\cup \partial_- M_0$ such that $\partial_+M_0\ne \emptyset$ and  contains  all the   boundary components with  exceptional fibres.  Such a decomposition of $\partial M_0$ determines a decomposition $\partial B_0=\partial_+ B_0 \cup \partial_ -B_0$, where  $\partial B_0$ denotes the boundary of the  surface and not of the orbifold (so including  the reflector arcs). Note that $\partial_+ B_0\ne \emptyset$. Referring to Figure~\ref{fig_main}, let $\Gamma_0$ be a graph embedded in $B_0$ such that (i) each vertex has at most valence three and (ii) $B_0\setminus (\Gamma_0\cup \partial_ -B_0)\cong (\partial_+ B_0\cap\partial_{\mathcal O}B_0)\times [0,1)$, where $\partial_{\mathcal O}B_0$ denotes the boundary of the orbifold\footnote{Note that  $\partial_{\mathcal O}B_0$ does not contain singular points except for the endpoints of   the  reflector arcs.} $B_0$. By the previous assumptions it is easy to see that $\Gamma_0$ does not reduce to a single point.

\begin{figure}[h!]                      
\begin{center}                         
\includegraphics[width=10cm]{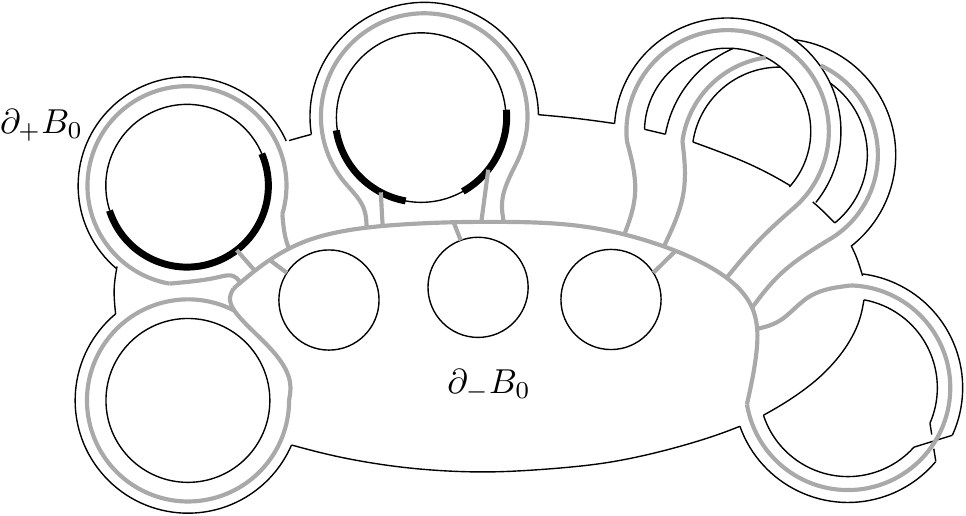}
\caption[legenda elenco figure]{The surface $B_0$ with the graph $\Gamma_0$ depicted in gray.}\label{fig_main}
\end{center}
\end{figure}

Let $P_0=f_0^{-1}(\Gamma_0)$. Since $\mathcal S\setminus\Gamma_0\cong \partial \mathcal S\times [0,1)$,  it follows that $M_0\setminus (P_0\cup\partial_- M_0)\cong\partial_+ M_0\times[0,1)$, and therefore $P_0$ is a skeleton for the \textit{main block} $(M_0,\partial_-M_0)$ without  true vertices. Note that,  since  $\Gamma_0$ intersects each  reflector arc in a single point,  $P_0$ intersects each component of  $AE(M)$ in an exceptional  fibre $\phi$ and $\overline{N}(\phi)\cap P_0$ is a M\"obius strip, where $\overline{N}(\phi)$ denotes a closed regular neighborhood of $\phi$  composed by fibres. Furthermore, $P_0\cap\partial_+M_0=\emptyset$ and $P_0$ intersects each component of $\partial_- M_0$ in a regular fibre. \\

\textbf{Exceptional block.} Let $M_E$ be either $N\times S^1$  or  $N\widetilde{\times} S^1$,  considered with the Seifert fibre space structures depicted in  Figure~\ref{example} $(a')$ and   $(b')$, respectively.  Denote by $f:M_E\to B_E$  the projection map. We represent  $N$ as in Figure~\ref{figanms} and $M_E$ as $(N\times I)/\sim$, where $\sim$ is an identification between $N\times \{0\}$ and $N\times\{1\}$ via the identity on $N$ if $M=N\times S^1$ and the composition of the reflection along the thick line with that  along the axis $\ell$ if $N\widetilde{\times} S^1$. Now let $N'\subset N$ be a closed regular neighborhood of the core of $N$ composed by fibres and disjoint from $\partial N$. Of course, $N'$ is a 
M\"obius strip and $N\setminus \textup{int}(N')\cong S^1\times I$. Referring to Figure~\ref{exc_block}, let $P_E\subset M_E$ be the polyhedron (depicted in gray) which is the union of:
\begin{itemize} 
\item the annulus $\alpha=(N\setminus \textup{int}(N'))\times \left\{\frac 14\right\}$;
\item the M\"obius strip $N'\times\{\frac 12\}$;
\item a band $\beta$ obtained by taking $\left(L\times I\right)/\sim$,  where $L\subset N'$ is the arc of the fixed points of the reflection along $\ell$, cutting it along $L\times \left\{\frac 12\right\}$ and pushing up (resp. pushing down) the part $L\times\{x\}$ with $x\ge \frac 12$ (resp. with $x\le \frac 12$) leaving fixed $L\times \{0\}\sim L\times\{1\}$, as shown in Figure~\ref{exc_block}. Observe that $\beta$ intersects transversally in an arc each $N'\times \{x\}$, with $x\neq \frac 12$, and intersect $N'\times \left\{\frac 12\right\}$ in two disjoint arcs;
\item the surface $\partial((N'\times I)/\sim)\setminus R$,   either a punctured torus or a punctured Klein bottle, where $R$ is the open dashed  2-cell depicted in Figure~\ref{fignewblock}. 
\end{itemize}
Note that $M_E\searrow ((N'\times I)/\sim)\cup \alpha$ and $((N'\times I)/\sim)\setminus P_E$ is a 3-ball. So, the polyhedron $P_E$ is a skeleton for the \textit{exceptional block} $(M_E,\emptyset)$  with only one true vertex (the 
thick point represented in both Figures ~\ref{exc_block}  and  \ref{fignewblock}). Moreover, $\partial_+M_E=\partial M_E$ and $P_E\cap \partial M_E$ is a regular fibre (i.e., $\alpha\cap \partial M_E$). 

\begin{figure}[h!]                      
\begin{center}                         
\includegraphics[width=7cm]{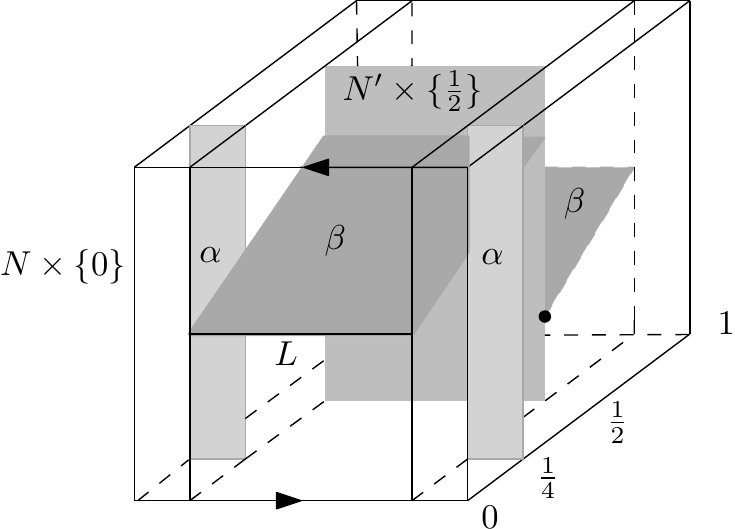}
\caption[legenda elenco figure]{The exceptional block $(M_E,\emptyset)$ and his skeleton $P_E$ (in gray).}\label{exc_block}
\end{center}
\end{figure}

\begin{figure}[h!]                      
\begin{center}                         
\includegraphics[width=2.5cm]{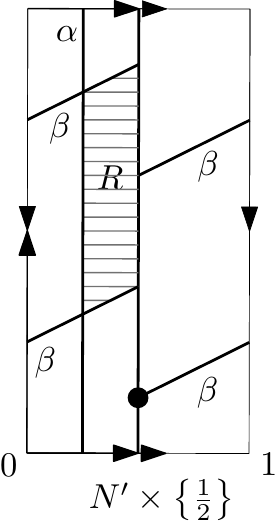}
\caption[legenda elenco figure]{The surface $\partial((N'\times I)/\sim)\setminus R$.}\label{fignewblock}
\end{center}
\end{figure}
We are ready to state our result on the complexity of bordered Seifert fibre spaces.

\begin{teo} \label{comp_bordo_reg} Let $M=\left\{b;\left(\epsilon,g,(t,k)\right);\left(h_1,\ldots,h_{m_+}\mid k_1,\ldots, k_{m_{-}}\right);
\left((p_1,q_1),\ldots,(p_r,q_r)\right)\right\}$  be a  Seifert fibre space    such that $\partial M\ne \emptyset$ (i.e., $m_++m_->0$). Then 
\begin{equation}
\label{bordo}
 c(M)\leq t +\sum_{j=1}^r \max\left\{S(p_j,q_j)-3,0\right\},
\end{equation}
where $S(p_j,q_j)$ denotes the sum of the coefficients of the expansion of $p_j/q_j$ as a continued fraction.

Moreover, if  $M=N\times S^1, N\widetilde{\times}S^1,D^2\times S^1, SK$ then   $c(M)=0$.
\end{teo} 

 \begin{proof}
We start by proving the last part of the statement. Referring to   Example~\ref{es_seifert_combi}, we have  $N\times S^1= \{0;(o_1,0,(1,0));(0\mid\ );\ \}$,  $N\widetilde{\times}S^1=\{0;(o,0,(1,1));(\ \mid0);\ \}$, $D^2\times S^1=\{0;(o_1,0,(0,0);(0\mid\ );(p,q)\}=\{0;(o_1,0,(0,0);(0\mid\ );\}$, $SK=\{0;(o_1,0,(0,0);(1\mid\ );\ \}$. A spine $S$ of the first two manifolds is the exceptional set (respectively,  a torus  and a Klein bottle), while a spine of the last ones is a circle (i.e., the axis of  the solid torus for $D^2\times S^1$  or any exceptional fibre   for $SK$). Indeed, in all cases $M\setminus S\cong\partial M\times[0,1)$ and so, since $S$ has no true vertices,  all these manifolds have complexity zero.
 
From now on, let $M$ be a Seifert fibre space different from the previous ones. Let $f:M\to B$  be the projection map and let $\partial_+B$ be the union of the boundary components of the underlying surface not corresponding to reflector circles. Denote by $B_0$ the surface obtained from $B$  by removing  disjoint open disks around each cone point and disjoint open collars around each reflector circle, clearly $B_0\subset B$ and $\partial_+ B\subset B_0$. Let $\partial_+B_0=\partial_+ B$, $\partial_-B_0=\partial B_0\setminus \partial_+B_0$ and $(M_0,\partial_- M_0)=(f^{-1}(B_0),f^{-1}(\partial_-B_0))$, therefore $\partial_+M_0=\partial M_0\setminus\partial_-M_0=\partial M$.
Since $\partial M\neq\emptyset$ and $M$ is neither $D^2\times S^1$ nor $SK$, then $(M_0,\partial_- M_0)$ is a main block. Moreover, $M\setminus \textup{int} (M_0)$ is the disjoint union of $t$  exceptional blocks $M_{E_i}$ (one for each component of $CE(M)$)  and $r$ fibred solid tori $T_1,\ldots, T_r$. Take the skeleton $P_0$ (without true vertices) for $M_0$ and the skeleton $P_{E_i}$ (with one true vertex) for each exceptional block. For $T_j$ we take the skeleton described 
in \cite{FW}, having $\max\left\{S(p_j,q_j)-3,0\right\}$ true vertices, for $j=1,\ldots,r$. By assembling $P_{E_i}$ with $P_0$ via the identity 
and the skeleton of $T_j$ with $P_0$ as described in \cite{FW}, we obtain the required result.
\end{proof}

Next corollary characterizes a wide class of Seifert fibre spaces having complexity zero. 
\begin{cor}
Let $M$ be a Seifert fibre space with $\partial M\ne\emptyset$,  and such that
\begin{itemize} 
\item[\textup{i)}] $SE(M)=AE(M)$ (i.e., $t=0$),
\item[\textup{ii)}] the $E$-fibres of $M$, if any, are of type  $(2,1)$, $(3,1)$  and $(3,2)$,
\end{itemize}
then $c(M)=0$.
\end{cor}
\begin{proof}
From the above conditions we have  $S(p_j,q_j)\leq 3$. So the statement follows directly 
from (\ref{bordo}).
\end{proof}
\end{subsection}

\begin{subsection}{The closed case}
Now we deal with the case $\partial M = \emptyset$. 
In the orientable case many results are already known: the complexity of $S^3$ is zero and in \cite[p.77]{M2} the following  upper bound for lens space complexity is given
\begin{equation}
c(L(p,q))\leq \max\{S(p,q)-3,0\},
\end{equation}
which has been proved to be sharp in many cases (see \cite{JRT,JRT2}).   Efficient upper bounds for all closed orientable Seifert fibre spaces have been obtained in \cite{MP2}, and in the following we will extend the main result of that paper to the non-orientable case.

As in the bordered case we  construct  a spine of $M$ by  assembling together the  skeletons of the different blocks in which $M$ is decomposed. Since the manifold is closed, we need to construct a skeleton for the space \linebreak $M_0=M \setminus N\left( E(M)\cup SE(M)\right)$\footnote{The regular neighborhood of $E(M)\cup SE(M)$ is supposed to be a union of fibres of $M$.} whose complement is a 3-ball, so we will need to  add a section of $f_{|_{M_0}}:M_0\to f(M_0)$ to the skeleton of the main block described in the case with non-empty boundary. \\

\textbf{Main block.} Let $M_0=\left\{0;\left(\epsilon,g,(0,0)\right);\left(0,\ldots,0 \mid 0,\ldots,0 \right);
\right\}$ be a Seifert fibre space without exceptional fibres and  let $f_0:M_0\to B_0$ be the projection map. Denote by $s=m_++m_-$  the number of boundary components of both $B_0$ and $M_0$. Set $\partial_- M_0 =\partial M_0$ and $\partial_+ M_0=\emptyset$.
Suppose that $B_0$ is neither a sphere nor a disk and denote by $\chi$ the Euler characteristic of the closed surface $B$  obtained from $B_0$ by capping off by disks all its boundary components (i.e.,   $\chi=2-2g$ if $\epsilon=o,o_1,o_2$ and $\chi=2-g$ if $\epsilon=n,n_1,n_2,n_3,n_4$). As a consequence, if $\chi=2$ then $s\ge 2$.

Let $D$ be a closed disk embedded in $\textup{int}(B_0)$ and let $A$ be the union of the disjoint arcs properly embedded in $B_0\setminus\textup{int}(D)$ described in Remark~\ref{costruzione} (depicted by thick lines in  Figure~\ref{fig_main3}). 
Then $A$ is non-empty and it is composed by $2-\chi$ edges with both endpoints in $\partial D$ and $s$ edges with an endpoint in $\partial D$ and the other in a component of $\partial B_0$. By construction $B_0\setminus\left(A\cup\partial B_0\cup D\right)$ is homeomorphic to an open disk, and therefore  $B_0\setminus\left(A\cup\partial B_0\cup \partial D\right)$ is the disjoint union of two open disks. Note that the number of points of $A$ belonging to $\partial D$ is at least two.

\begin{figure}[h!]                      
\begin{center}                         
\includegraphics[width=5.5cm]{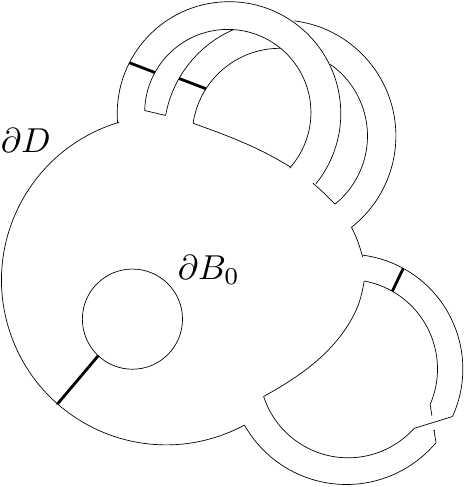}
\caption[legenda elenco figure]{The set $A\subset B_0\setminus\textup{int}(D)$.}\label{fig_main3}
\end{center}
\end{figure}

If  $s_0: B_0\to M_0$   is a section of $f_0$, then  $P''_0=s_0(B_0)\cup f_0^{-1}(A)\cup f_0^{-1}(\partial D)$ is a non-simple  polyhedron  since $\textup{int}(s_0(A))$ is a collection of quadruple lines in the  polyhedron (the link of each point is homeomorphic to a graph with two vertices and four edges connecting them), and a similar phenomenon occurs for $s_0(\partial D\setminus A)$.
In order to make the polyhedron $P''_0$ simple we perform  ``small'' shifts by moving in parallel the disk $s_0(D)$ along  the fibration and the components of $f_0^{-1}(A)$ as depicted  in Figure~\ref{freedom}. As shown by the pictures, the shift of any component of  $f_0^{-1}(A)$ may be performed in two different ways that, as we will see, are not usually equivalent in term of complexity of the final spine. On the contrary, the two possible parallel shifts for $s_0(D)$ are equivalent as it is evident from Figure~\ref{alette}, which represents the torus $T=f_0^{-1}(\partial D)$. It is convenient to think the shifts of $f_0^{-1}(A)$ as performed on the components of $A$.

\begin{figure}[h!]                      
\begin{center}                         
\includegraphics[width=13.5cm]{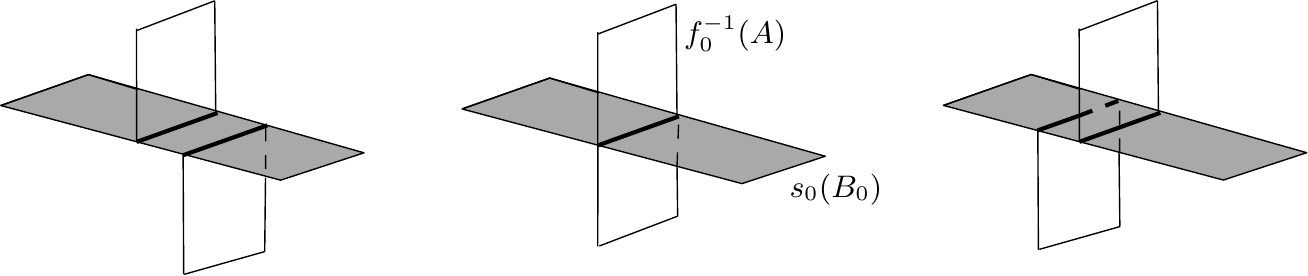}
\caption[legenda elenco figure]{The two possible shifts on a component of $f_0^{-1}(A)$.}\label{freedom}
\end{center}
\end{figure}

Let $P'_0=s_0(B_0\setminus\textup{int}(D))\cup D'\cup W \cup f_0^{-1}(\partial D)$ be the polyhedron obtained from $P''_0$ after the shifts, where $D'$ and $W$ are the results of the shifts of  $s_0(D)$ and  $f_0^{-1}(A)$, respectively. It is easy to see that $P'_0\cup\partial M_0$ is simple and $P'_0$ intersects each component of $\partial M_0$ in a non-trivial theta-curve. Moreover, $P'_0$ has 3 true vertices for each point of $A\cap\partial D$, so it has exactly $12-6\chi+3s$ true vertices.  Since  $M_0\setminus\left( P'_0\cup\partial_- M_0\right)$ is the disjoint union of two open balls, in order to obtain a skeleton $P_0$  for   the \textit{main block} $(M_0,\partial_-M_0)$ it is enough to remove a suitable open 2-cell from the torus $T=f_0^{-1}(\partial D)\subset P'_0$,  connecting in this way the two balls. Since  $\Gamma=T \cap  (s_0(B_0\setminus\textup{int}(D))\cup D'\cup A)$ is a  graph cellularly embedded in $T$ whose vertices are true vertices of $P'_0$, we will remove the region $R$ of $T\setminus \Gamma$ having in the boundary the highest number of vertices of $\Gamma$.

Referring to Figure \ref{alette},  the graph $\Gamma$ is composed by two horizontal parallel loops  $d=\partial (s_{0}(D))$ and $d'=\partial D'$, and an arc with both  endpoints on  $d$ for each boundary point of $A$ belonging to $\partial D$.  Reversing the shift of a component of $A$ performs a symmetry along $d$ of the correspondent arc(s). Clearly, if the shift is performed on a component of $A$ which is the cocore of a handle, both   arcs corresponding to the endpoints change as just described.  Moreover, if the core of an orientable (resp. non-orientable) handle is  sent by $\omega$  to $1$ then the corresponding two arcs  (which are not necessarily  consecutive in $\Gamma$, as suggested by the dots in the pictures) are as  in picture (a) (resp. (b)), or in the mirrored ones with respect to $d$. On the contrary, if  the   core is sent to  $-1$  then the rightmost arc in each picture should be symmetrized with respect to the loop $d$.

\begin{figure}[h!]                      
\begin{center}                         
\includegraphics[width=12cm]{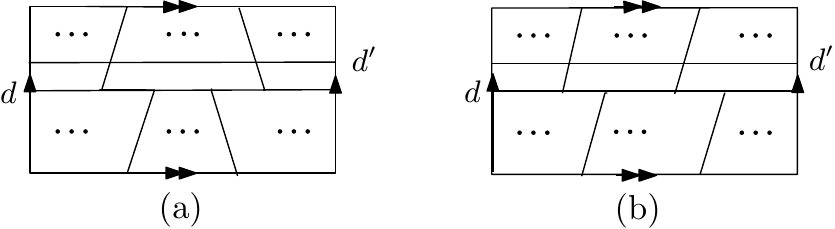}
\caption[legenda elenco figure]{A fragment of the graph $\Gamma$ embedded in $f_0^{-1}(\partial D)$.}\label{alette}
\end{center}
\end{figure}

A region of $T\setminus \Gamma$ has 5  vertices when the arcs belonging to its boundary are parallel and either 4 or 6 vertices otherwise. Since all regions has 5 vertices if and only if all arcs are parallel,  the shifts of the elements of $A$ can be chosen in such a way that there exists a region with 6 vertices in all cases except when $\chi=1$, $\epsilon=n_1$ and $s=0$. This exceptional case is the one depicted in Figure~\ref{alette} (b) without   dots: in that case all regions have 5 vertices. By removing  such a region  from $P'_0$ we obtain a polyhedron $P_0$ for the main block $(M_0,\partial_- M_0)$ with  $6(1-\chi)+3s$ true vertices, while in the special case $\chi=1$, $\epsilon=n_1$ and $s=0$, the polyhedron has exactly one true vertex.
We remark that changing the shift of a component of $A$ intersecting $\partial B_0$  changes the intersection between the corresponding element of $f^{-1}(A)$ and $\partial_-(M_0)$ (which is a non-trivial theta curve) by a flip move (see bottom and top face of the block of Figure \ref{flip}).

It  is important to point out that when $s=0$ the polyhedron $P_0$ is a simple spine for $M_0$.   \\

\begin{figure}[h!]                      
\begin{center}                         
\includegraphics[width=4cm]{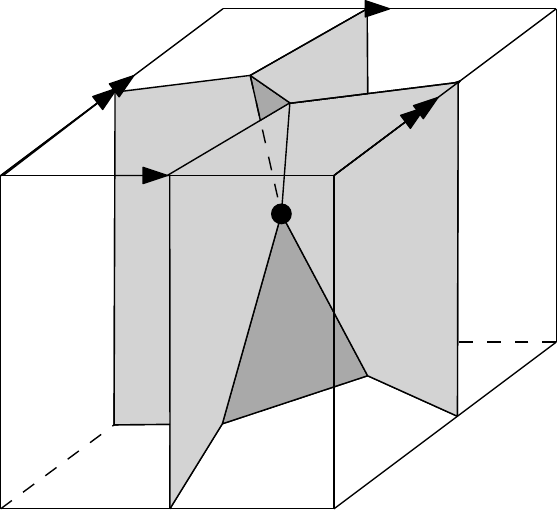}
\caption[legenda elenco figure]{A flip block connecting two theta graphs.}\label{flip}
\end{center}
\end{figure}

\textbf{Exceptional block.} 
Let $M_E$  be either  $N\times S^1$  or  $N\widetilde{\times} S^1$  considered with the Seifert fibre space structures depicted in  Figure~\ref{example} $(a')$ and   $(b')$, respectively, and  denote by $f:M_E\to B_E$  the projection map. 
Consider the skeleton $P_E$ of the exceptional block  $(M_E,\emptyset)$ constructed in the bordered case (see Figure~\ref{exc_block}). In that case $P_E\cap\partial M_E$ is a regular fibre, while in order to make the assembling with the main block the intersection should be a theta graph. Therefore, referring to Figure~\ref{exc_block2}, we modify $P_E$  as follows:
\begin{itemize}
 \item  add an annulus $\gamma$ which is a section of the restriction of $f$ to the space $M_E\setminus \textup{int}\left(\left(N'\times I\right)/\sim\right)$;
 \item  modify the annulus $\alpha$ by\footnote{We perform this change in order to remove the quadruple line $\alpha\cap\gamma$.}  pushing (i) the bottom part of the right strip toward $N\times\{0\}$ and the upper part of the left strip toward $N\times \{1\}$;
  \item take the surface  $\partial((N'\times I)/\sim)\setminus R$, where the 2-cell $R$ is the  dashed region of the left picture of  Figure~\ref{fignewblock2}.  
\end{itemize}

If $P_E$ still denote the resulting skeleton, then  $M_E\searrow \left((N'\times I)/\sim\right)\cup \alpha\cup\gamma$ and $((N'\times I)/\sim)\setminus P_E$ is a 3-ball. Therefore the polyhedron $P_E$ is a skeleton with 3 true vertices (the thick points represented both in Figures~\ref{exc_block2}  and \ref{fignewblock2}) for the \textit{exceptional block} $(M_E,\emptyset)$. Note that if we modify $\alpha$ in the opposite way, namely
pushing (i) the upper part of the right strip toward $N\times\{0\}$ and the bottom part of the left strip toward $N\times \{1\}$ the theta graph on $P_E\cap \partial M_E$ changes by a flip. Anyway,  we can still find a skeleton with $3$ true vertices (see the right part of Figure~\ref{fignewblock2}). 

\begin{figure}[h!]                      
\begin{center}                         
\includegraphics[width=7cm]{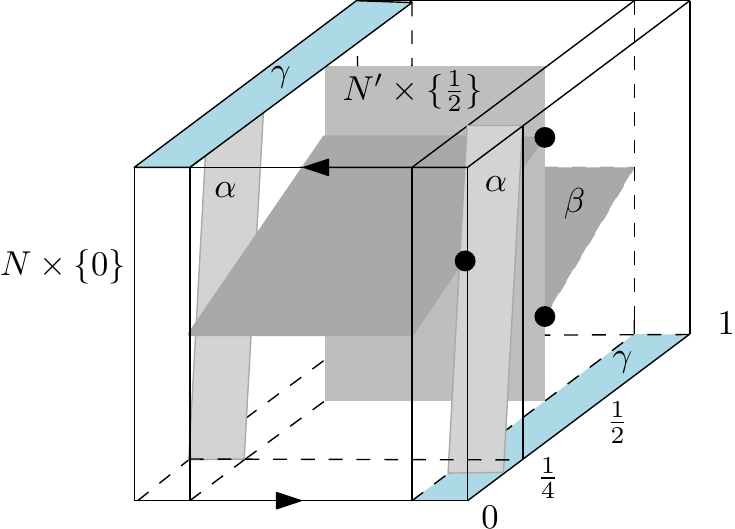}
\caption[legenda elenco figure]{The exceptional skeleton $P_E$ (in gray) for the block $(M_E,\emptyset)$.}\label{exc_block2}
\end{center}
\end{figure}

\begin{figure}[h!]                      
\begin{center}                         
\includegraphics[width=6.5cm]{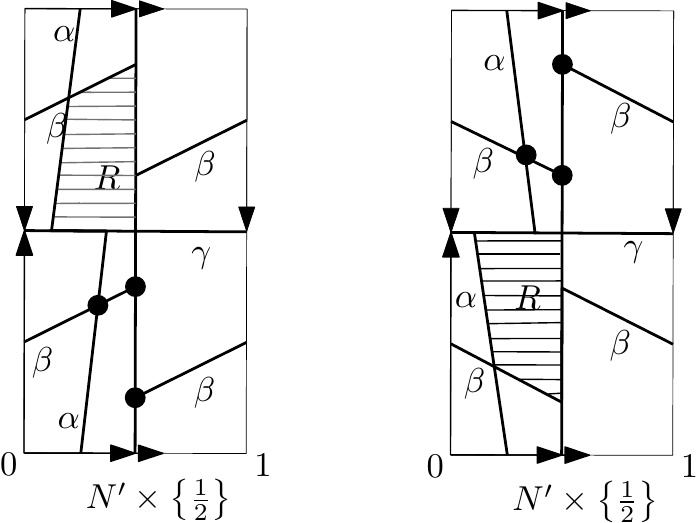}
\caption[legenda elenco figure]{The surface $\partial((N'\times I)/\sim)\setminus R$, corresponding  to the two different choices for $P_E$.}\label{fignewblock2}
\end{center}
\end{figure}

Now we are ready to state our result on the complexity of closed Seifert fibre spaces.

\begin{teo}\label{closed} Let $M=\left\{b;\left(\epsilon,g,(t,k)\right);\left(\ \mid \ \right);
\left((p_1,q_1),\ldots,(p_r,q_r)\right)\right\}$  be a closed  Seifert fibre space with $b\ge -r/2$, and let $\chi = 2-2g$ if $\epsilon=o,o_1,o_2$ and $\chi=2-g$ if $\epsilon=n,n_1,n_2,n_3,n_4$.
\begin{enumerate}
 \item If $\chi=2$ and $r=t=0$, then $c(M)\leq \max\{b-3,0\}$;
 \item if $\chi=2$, $t=0$, $r=1$ and $b>0$, then   $c(M)\leq \max\{b+S(p_1,q_1)-3,0\}$;
 \item if $\chi=2$, $t=0$, $r=1$ and $b=0$, then   $c(M)\leq \max\{S(p_1,q_1)-3-\lfloor p_1/q_1\rfloor,0\}$, where $\lfloor \cdot \rfloor$ denotes the integer part function;
 \item if $\chi=1$, $\epsilon=n_1$, $r=t=0$ and $b=0$, then   $c(M)\leq 1$;
  \item if $\chi=1$, $\epsilon=n_1$, $r=t=0$ and $b=1$, then   $c(M)=0$;
  \item in all  other cases:
  \begin{equation}\label{petronio}
c(M)\leq \max\{b-1+\chi,0\}+6(1-\chi)+\sum_{j=1}^r \left(S(p_j,q_j)+1\right), 
\end{equation} 
if $M$ is orientable\footnote{Formula (\ref{petronio}) has been introduced in \cite{MP2}. Here we give a new and more direct proof of it.};
\begin{equation}\label{sharm}
c(M)\leq 6(1-\chi)+6t+\sum_{j=1}^r \left(S(p_j,q_j)+1\right), 
\end{equation}
if $M$ is non-orientable.
\end{enumerate}
\end{teo} 
 \begin{proof}
\begin{enumerate}
 \item In this case $M=L(b,1)$ (see \cite{O}), so the result follows from (2). 
 \item In this case $M=L(bp_1+q_1,p_1)$ (see \cite{O}), so the result follows from (2). 
 \item In this case $M=L(q_1,p_1)$ (see \cite{O}), so the result follows from (2) since $M=L(q_1,r_1)$, where $r_1\equiv p_1 \textrm{ mod } q_1$ with $0<r_1<q_1$, and $S(q_1,r_1)=S(p_1,q_1)-\lfloor p_1/q_1\rfloor$. 
 \item In this case $M=\mathbb{RP}^2\times S^1$ and a spine for $M$ is the polyhedron $P_0$ of the main block, which in this case has exactly one true vertex. 
 \item In this case $M=S^2\widetilde\times S^1$ (see \cite{O}). Let  $S^2\widetilde\times S^1=(S^2\times I)/\sim$, where $(x,0)\sim (a(x),1)$ being $a$ the antipodal map of $S^2$, then an almost spine for $M$ is $((S^2\times\{0\})\cup(\{P\}\times I))/\sim$, which is homeomorphic to a 2-sphere with a diameter, and therefore has no true vertices.\\
\end{enumerate} 

Now  we turn to the proof of formulae~(\ref{petronio}) and (\ref{sharm}). 

If $\chi=2$, $t=1$ and $r=0$, then the base space is a disk whose boundary is a reflector circle. A  simple spine for $M$\footnote{It is easy to see that in this case  $M=S^2\widetilde\times S^1$.}  is the union of the exceptional set (which is a torus  $T$) and a section of the fibration (which is a disk $D$), being $T\cap D=\partial D$ a non-trivial simple closed curve on $T$. Of course the spine has no true vertices and therefore $c(M)=0$, which proves (\ref{sharm}).

From now on, let $M$ be a Seifert fibre space  different from the previous ones. 

First suppose $b=0$. Let $f:M\to B$ be the projection map and denote by $B_0\subset B$ the surface obtained from $B$  by removing  disjoint open disks around each cone point and disjoint open collars around each reflector circle. Let $\partial_+B_0=\emptyset$, $\partial_-B_0=\partial B_0$ and $(M_0,\partial_- M_0)=(f^{-1}(B_0),f^{-1}(\partial_-B_0))$, therefore $\partial_+M_0=\emptyset$.
The block  $(M_0,\partial_- M_0)$ is a main block with $s=t+r$  boundary components and $M\setminus \textup{int} (M_0)$ is the disjoint union of $t$  exceptional blocks $M_{E_i}$ (one for each component of $CE(M)$) and $r$ fibred solid tori $T_1,\ldots, T_r$ (one for each isolated exceptional fibre).

For $M_0$ we take a skeleton $P_0$ as previously described, where the choice of the shifts depends on the following.
The  skeleton for $T_j$ described in \cite{FW} and  having $S(p_j,q_j)-3$ true vertices, have to be modified since in the closed case $P_0\cap T_j$ is a theta curve, and  no longer a simple closed curve. So we replace a transitional block (having no vertices) connecting a regular fibre with a theta graph denoted by $\theta'$ (according to the notation of \cite{FW}),  with either one or two flip blocks (see Figure \ref{flip}), each having one true vertex, connecting  the theta graph $P_0\cap T_j$  to $\theta'$. The number of the additional flip blocks depends on the shift of the corresponding component $\delta$ of $A$ used in the construction of the main block. We call the shift of $\delta$ {\it regular} when a single flip is sufficient and {\it singular} when two flips are required  (see Figure \ref{A}, where the shifted arcs are denoted by dotted lines). Since we want to have a skeleton $P_0$ for $M_0$ with $6(1-\chi)+3t+3r$  true vertices,  all flips can be chosen regular if either $t>0$  or $\chi<2$ and  all flips except one can be chosen regular otherwise (see Figure \ref{B}).

For $M_{E_i}$  we take a skeleton $P_{E_i}$ such that the theta graph $P_{E_i}\cap M_{E_i}$ coincides with the theta graph on the corresponding component of $M_0\cap P_0$ for $i=1,\ldots ,t$.  The skeleton has always 3 true vertices, since the  choice of the shift on the corresponding component of $A$ does not affect the number of its true vertices (see Figure \ref{fignewblock2}).

\begin{figure}[h!]                      
\begin{center}                         
\includegraphics[width=5cm]{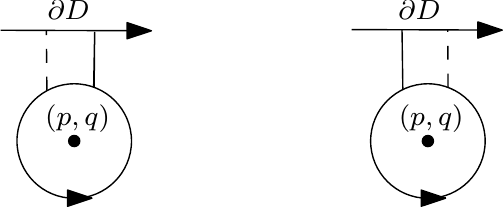}
\caption[legenda elenco figure]{Regular shift (on the left) and singular shift (on the right).}\label{A}
\end{center}
\end{figure}

\begin{figure}[h!]                      
\begin{center}                         
\includegraphics[width=12cm]{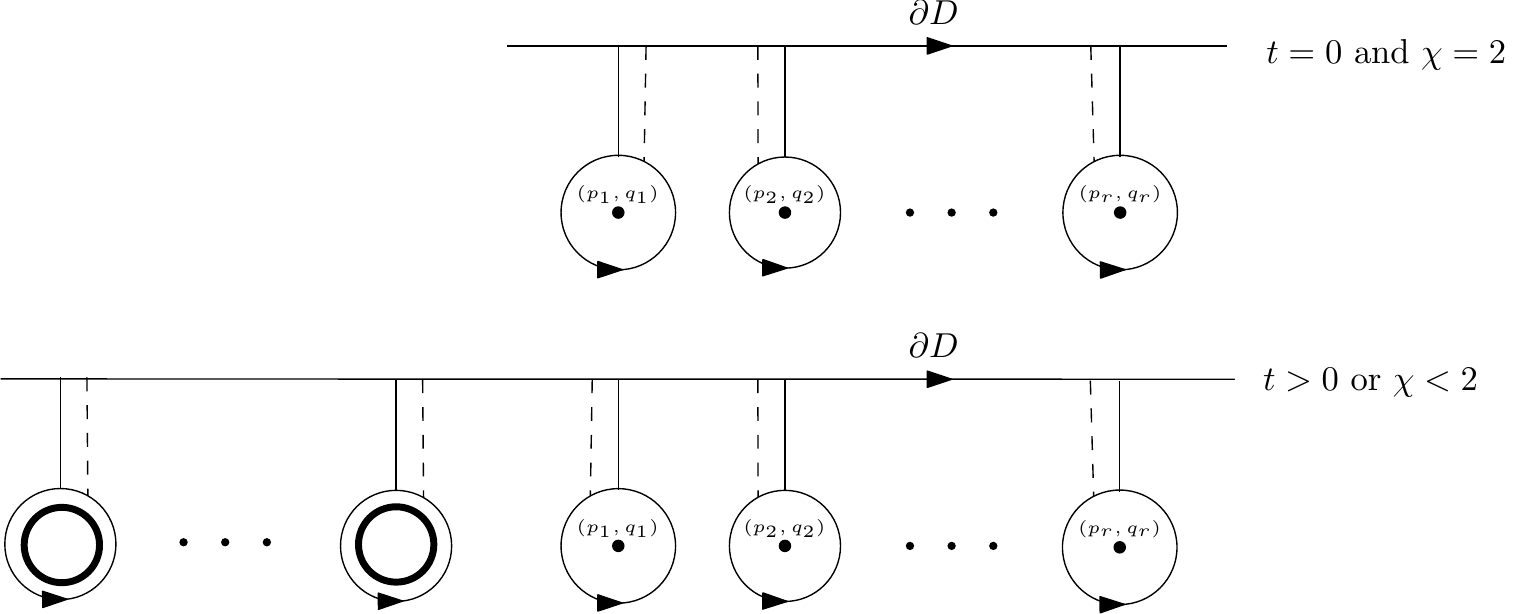}
\caption[legenda elenco figure]{The choice of the shifts for the  components of $A$ intersecting $\partial D$.}\label{B}
\end{center}
\end{figure}

By assembling  $P_{E_i}$  and the skeleton of $(T_j,\emptyset)$ with $P_0$ by the identity for $i=1,\ldots,t$ and $j=1,\ldots,r$,    we obtain the desired spine $S$ for $M$. When either $t>0$ or $\chi<2$,
the number of true vertices of $S$ is  $6(1-\chi)+3(t+r)+3t+\sum_{j=1}^r \left(S(p_j,q_j)-2\right)$, which proves (\ref{petronio}) and (\ref{sharm}). When $\chi=2$ and $t=0$, the number of true vertices of $S$ is $-6+3r+\sum_{j=1}^r \left(S(p_j,q_j)-2\right)+1$, and~(\ref{petronio}) is proved. \\

Let now $b\neq 0$ (and therefore $t=0$). We prove (\ref{petronio}) and (\ref{sharm}) in different steps: $b=1$, $b=-1$ and $\vert b\vert >1$. Let $M'$ be the  Seifert fibre space with the same parameters of $M$ but with $b=0$, and let $S'$ be the spine of $M'$ constructed as before.

\begin{figure}[h!]                      
\begin{center}                         
\includegraphics[width=9cm]{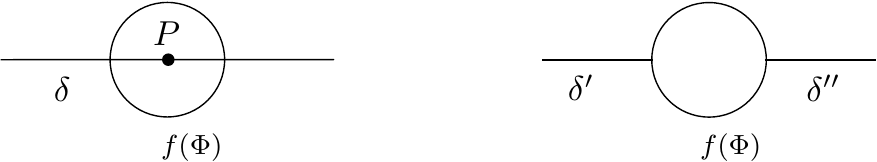}
\caption[legenda elenco figure]{}\label{quattro}
\end{center}
\end{figure}

First suppose $b=1$. In this case $M$ can be obtained from $M'$  by adding a non-trivial (non-exceptional) fibre of type $(1,1)$. Namely,  by removing from $M'$ a trivially fibred solid torus $\Phi$ which is a fiber-neighborhood of a regular fibre $\phi$, and by attaching back a solid torus $ST=\Delta\times S^1$  via a homeomorphism $\psi:\partial(ST) \to \partial \Phi$  such that $\psi(\partial \Delta\times \{1\})$ is  a curve of type $(1,1)$ on $\partial \Phi$. It is convenient to take the fibre $\phi$ corresponding to an internal point $P$ of $A$ and suppose that $f(\Phi)$ is a "small" disk intersecting the component $\delta$ of $A$ containing $P$ in an interval and  being disjoint from $\partial B_0\cup\partial D$ and from the other components of $A$. In this way $\delta\setminus \textrm{int}(f(\Phi))$ is the disjoint union of two arcs $\delta'$ and $\delta''$, where we can perform the shifts independently (see Figure~\ref{quattro}). A spine $S$ of $M$ is obtained as follows. 

First of all, remove from the spine $S'$ of $M'$ the internal part of $\Phi$ and possibly change the twist corresponding to either $\delta'$ or $\delta''$ without increasing the number of true vertices of the main block. 
Let $S''$ be the polyhedron obtained in this way and set $S'''=S''\cup \partial\Phi\cup_{\psi}\left(\Delta\times \{1\}\right)$. Then $M\setminus S'''$ is the union of two open 3-balls, since $\Phi\setminus \left(\partial\Phi\cup_{\psi}(\Delta\times \{1\})\right)$ is an open 3-ball. 
Therefore, in order to obtain the spine $S$ of $M$ we have to remove from $S'''$ a suitable open 2-cell on $\partial\Phi$. The space $\Gamma'=(\partial\Phi\cap S'')\cup\psi(\partial\Delta\times \{1\}))$ is a graph cellularly embedded in $\partial\Phi$ (see Figure~\ref{C}, where the  label $1$ inside the disc stands for the fibre type $(1,1)$), so we delete the region $R$ of $\partial\Phi\setminus \Gamma'$ having in the boundary the highest number of vertices of $\Gamma'$. If we take for $\delta'$ and $\delta''$ the shifts induced by the one of $\delta$, then we can choose $R$ containing in its boundary all the  true vertices of $S$ belonging to $\partial \Phi$ with the exception of one (see the first two pictures of Figure~\ref{C}) and $S$ has one true vertex more than $S'$. On the contrary, if one of the two shifts is changed as in the third draw of Figure~\ref{C}, then $R$ can be chosen containing in its boundary all true vertices and therefore $S$ and $S'$ have the same number of true vertices. 

So, if either $\chi=2$ (and therefore $r\ge 2$) or $\chi=1$ and $\epsilon=n_2$, we take as $\delta$ any arc of $A$ and use for $\delta'$ and $\delta''$ the shifts induced by the one of $\delta$. Then (\ref{petronio}) is proved. 

\begin{figure}[h!]                      
\begin{center}                         
\includegraphics[width=12cm]{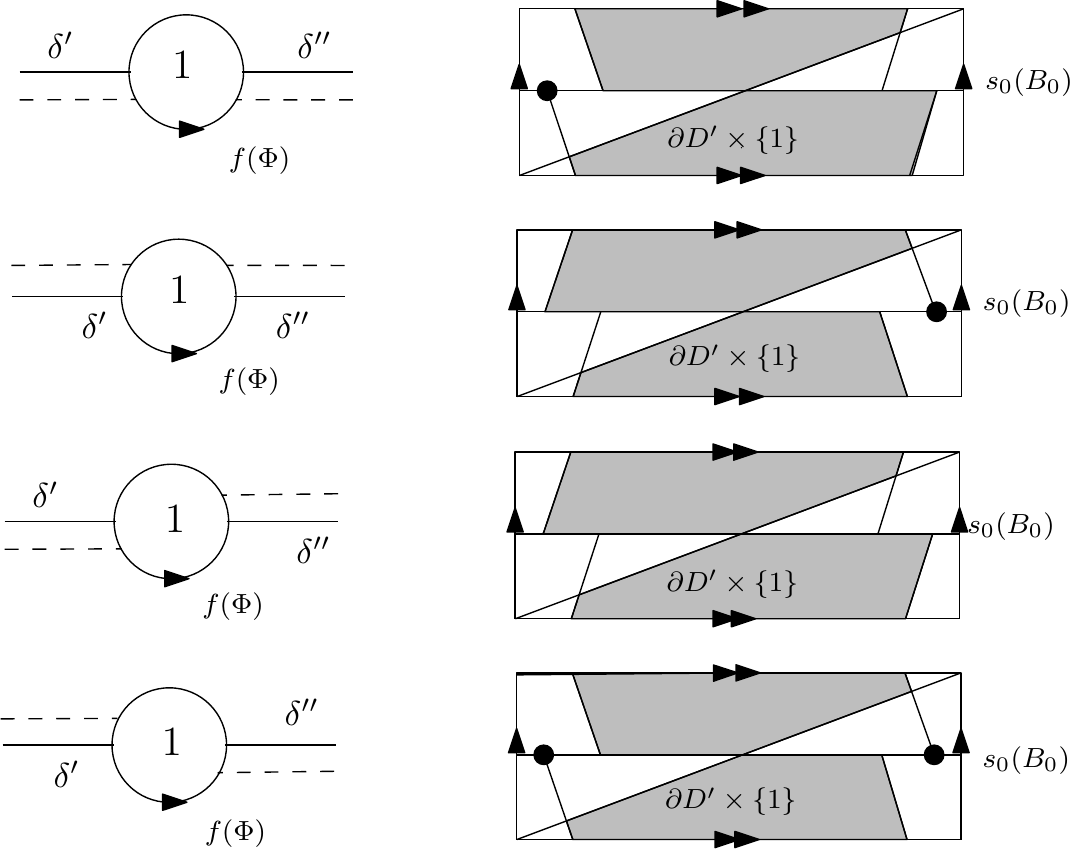}
\caption[legenda elenco figure]{The graph $\Gamma'$ embedded in the torus $\partial \Phi$ with different choices of the shifts for $\delta'$ and $\delta''$.}\label{C}
\end{center}
\end{figure}

If $\chi<2$ and $\epsilon=n_1$ when $\chi=1$, it is always  possible to choose an arc $\delta$ of $A$ not intersecting $\partial B_0$ and to choose the shifts for $\delta'$ and $\delta''$ as depicted in  the third draw of Figure~\ref{C} without  increasing the number of true vertices of the main block. Then (\ref{petronio}) and (\ref{sharm}) are proved.

\begin{figure}[h!]                      
\begin{center}                         
\includegraphics[width=12cm]{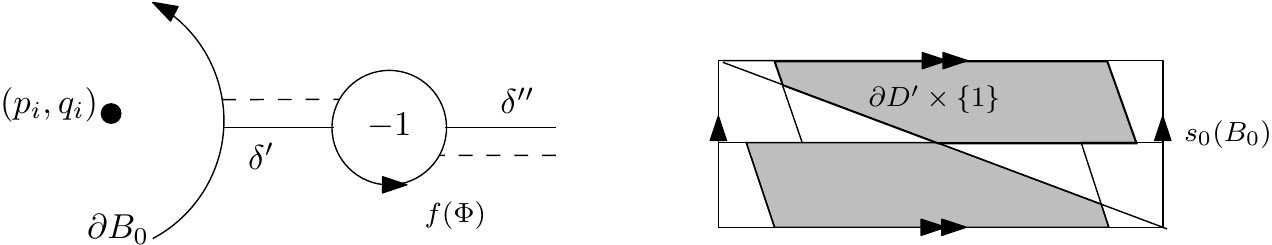}
\caption[legenda elenco figure]{}\label{tre}
\end{center}
\end{figure}

In this way~(\ref{sharm}) is proved for all cases.\\

Let now $b=-1$ (and therefore $r\ge 2$). The procedure to obtain $M$ from $M'$ and to construct the spine  $S$  is the same as in  case $b=1$, but this time  adding a non-trivial fibre of type $(1,-1)$. 

If $\chi=2$ take $\delta$ as the arc with non-regular shift. Then 
the shift of $\delta'$ and $\delta''$ can be chosen as in Figure~\ref{tre} (no true vertices out of the boundary of the gray region). Since the shift of the new arc which intersect $\partial B_0$ (say $\delta'$) becomes regular, the spine $S$ has one true vertex less than $S'$ (namely it has $-6+3r+\sum_{j=1}^r \left(S(p_j,q_j)-2\right)$ true vertices) and~(\ref{petronio}) is proved.

\begin{figure}[h!]                      
\begin{center}                         
\includegraphics[width=9cm]{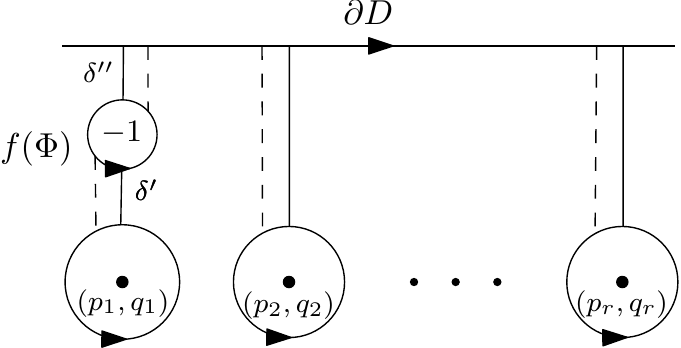}
\caption[legenda elenco figure]{}\label{E}
\end{center}
\end{figure}

If $\chi<2$ take as  $\delta$ an arc intersecting $\partial B_0$ and having a regular shift. Then the shifts of $\delta'$ and $\delta''$ can be chosen as in Figure \ref{E}: the main block does not increases the number of true vertices and the spine has the same number of vertices of the one of $M'$, so  proving~(\ref{petronio}).

Finally, let $\vert b\vert>1$. In this case $M$ can be obtained from $M'$ by replacing $\vert b\vert$ trivial fibres with $\vert b\vert$ fibres of type $(1,\textrm{sign}(b))$. Again it is convenient to choose the fibres corresponding to internal points of $A$ and  to remove disjoint fibre-neighborhoods of the chosen fibres with the same properties as before.

If $b<-1$ take all points in different arcs $\delta_i$ which intersect $\partial B_0$ (it is possible since $r\ge -2b=2\vert b\vert>\vert b\vert$) and the shifts of the new arcs as depicted in Figure \ref{G3} (so $\partial \Phi_i$ is as in  Figure \ref{tre}). Moreover, if $\chi=2$ the first point has to be taken in the arc with non-regular shift. In this way the shifts of all new arcs still intersecting $\partial B_0$ (say $\delta''_i$) are regular, and the number of true vertices of the main block does not increase. 

Therefore the spine $S$ has the same number of true vertices of the case $b=-1$, which proves~(\ref{petronio}).

\begin{figure}[h!]                      
\begin{center}                         
\includegraphics[width=8cm]{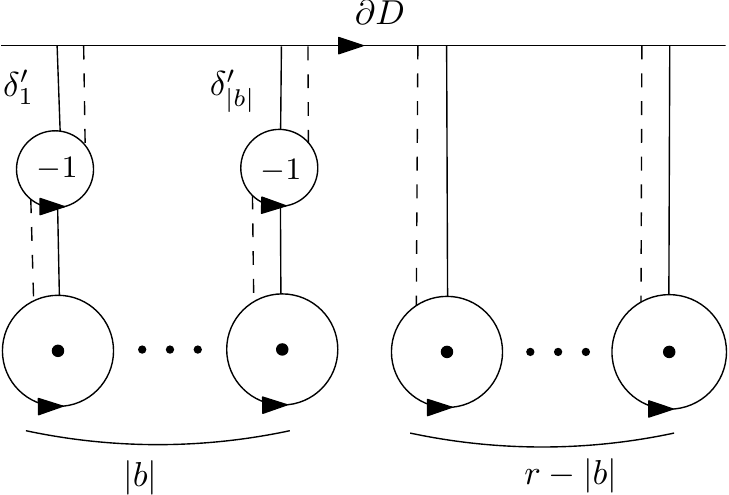}
\caption[legenda elenco figure]{}\label{G3}
\end{center}
\end{figure}

If $b>1$ then it is possible to take $1-\chi$ points in different arcs $\delta_i$ not intersecting $\partial B_0$ and to choose the shifts of $\delta'_i$ and $\delta''_i$ in such a way that (i) $\partial \Phi_i$ is as in the third draw of Figure \ref{C} and (ii) the number of true vertices of the main block does not increase (see the upper picture of   Figure \ref{F}). The remaining $b-1+\chi$ points are chosen outside $f(\Phi_i)$ for all $i$, with the shifts of the new edges  induced by those of the old ones as depicted in the bottom part of Figure \ref{F}. In this way~(\ref{petronio}) is proved.

\begin{figure}[h!]                      
\begin{center}                         
\includegraphics[width=12.5cm]{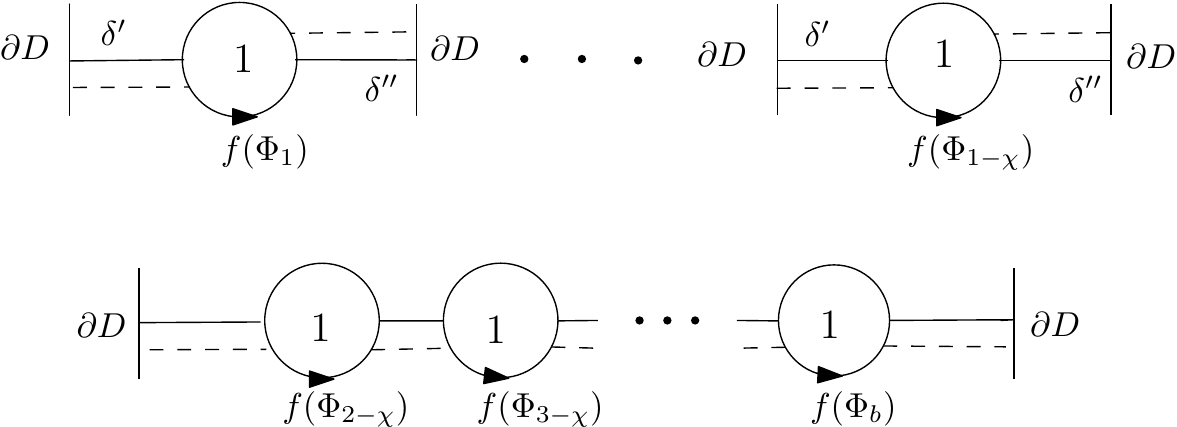}
\caption{}
\label{F}
\end{center}
\end{figure}

\end{proof}

\begin{oss}
\label{fine}
In the non-orientable (closed) case, exact values of complexity are listed in \cite{AM} (up to complexity 7), in \cite{B}  (up to complexity 10) and at the web page \texttt{https://regina-normal.github.io}  (up to complexity 11).  For all the Seifert fibre spaces included in those lists, that are about 350,   the complexity estimation given by   (\ref{sharm}) is  sharp, except in the  cases: (i) $\chi=1$, $\epsilon=n_1$, $t=0$, $r=1$ and (ii) $\chi=2$, $t=r=1$. Note that these cases  concern different Seifert fibre structures of $\mathbb{RP}^2\times S^1$ and  $S^2\widetilde{\times}S^1$, respectively, whose correct estimation is given in 4. and 5. of Theorem~\ref{closed}.

It is worth noting that Burton in \cite{B} and \texttt{https://regina-normal.github.io} uses in some cases non-normalized parameters for Seifert fibre spaces:  the space  $\{1; (\epsilon, g, (0,0)); (\ \mid\ ); ((p_1,q_1),\dots,(p_{r-1},q_{r-1}), (p_r,q_r)) \}$, with  $\epsilon\in\{o_2,n_1,n_3,n_4\}$, appears  there as  $\{0; (\epsilon, g, (0,0)); (\ \mid\ ); ((p_1,q_1),\dots,(p_{r-1},q_{r-1}), (p_r,p_r-q_r))\}$.

In  the orientable (closed) case,  exact values of complexity  are listed in \cite{M2} (up to complexity 6 and partially up to complexity 7),  in \cite{MP2} (up to complexity 6 and partially up to complexity 9) and in the web page  \texttt{http://matlas.math.csu.ru/?page=search} (up to complexity 12). For all the Seifert fibre spaces included in those lists the complexity estimation given by (\ref{petronio})  is sharp except for the following cases:

\begin{itemize}
   \item[(i)] manifolds of the form  $M=\left\{-1;\left(o_1,0,(0,0)\right);\left(\ \mid \ \right); \left((2,1),(n,1),(m,1)\right)\right\}$ with $2\leq n\leq m$, where the estimation of~(\ref{petronio}) exceeds the exact value by  one or two.\footnote{These manifolds are the ones of the family $\mathcal M^*$ studied in \cite{MP2}, where an estimation of the complexity sharper than (\ref{petronio}) is given.} In particular, if $n=2$ then $M$ also admits the Seifert fibre structure $M=\left\{m;\left(n_2,1,(0,0)\right);\left(\ \mid \ \right); \right\}$ (see \cite{O}), and in this case~(\ref{petronio}) gives the sharper value of complexity $c(M)\le m$;
   
\item[(ii)] manifolds of the form  $M=\left\{-1;\left(o_1,0,(0,0)\right);\left(\ \mid \ \right); \left((2,1),(3,1),(p,q)\right)\right\}$, with $p/q>5$ and $p/q\notin \mathbb Z$, where the estimation of~(\ref{petronio}) exceeds the exact value by  one.
   
\end{itemize}

\end{oss}
The sharpness of formula \eqref{sharm} in all  known cases justifies the following conjecture.

\begin{con} Let $M=\left\{b;\left(\epsilon,g,(t,k)\right);\left(\ \mid \ \right);
\left((p_1,q_1),\ldots,(p_r,q_r)\right)\right\}$  be a non-orientable closed irreducible and $\mathbb P^2$-irreducible Seifert fibre space, then 
$$
c(M)= 6(1-\chi)+6t+\sum_{j=1}^r \left(S(p_j,q_j)+1\right).
$$

\end{con}

\end{subsection}

\end{section}

\bigskip

\begin{flushleft}

{\bf AMS Subject Classification: Primary  57M27; Secondary 57N10, 57R22.}\\[2ex]

\vbox{

Alessia~CATTABRIGA\\
Department of Mathematics, University of Bologna\\
Piazza di Porta San Donato 5, 40126 Bologna, ITALY\\
e-mail: \texttt{alessia.cattabriga@unibo.it}\\}

\medskip

Sergei~MATVEEV\\
Chelyabinsk State University and Krasovskii Institute of Mathematics and Mechanics\\
129 Bratiev Kashirinykh st., 454001 Chelyabinsk, RUSSIA\\
e-mail: \texttt{svmatveev@gmail.com}\\

\medskip

Michele~MULAZZANI\\
Department of Mathematics and ARCES, University of Bologna\\
Piazza di Porta San Donato 5, 40126 Bologna, ITALY\\
e-mail: \texttt{michele.mulazzani@unibo.it}\\

\medskip

Timur~NASYBULLOV\\
Department of Mathematics, Katholieke Universiteit Leuven Kulak\\
Etienne Sabbelaan 53, 8500 Kortrijk, BELGIUM\\
e-mail: \texttt{timur.nasybullov@mail.ru}\\[2ex]

\end{flushleft}


\footnotesize

\begin{thebibliography}{99}
\bibitem[AM]{AM} \textsc{G. Amendola, B. Martelli}, Non-orientable manifolds of complexity up to 7, \textit{Topology Appl.}, \textbf{150} (2005), 179--195.

\bibitem[B]{B} \textsc{B. A. Burton}, Enumeration of Non-Orientable 3-Manifolds Using Face-Pairing Graphs and Union-Find, \textit{Discrete Comput. Geom.}, \textbf{38} (2007) 527--571.

\bibitem[E]{E} \textsc{D. B. A. Epstein}, Periodic flows on 3-manifolds, \textit{Ann. of Math.}, \textbf{95} (1972), 66--82.  

\bibitem[Fi]{Fi} \textsc{R. Fintushel},  Local $S^1$-action on 3-manifolds, \textit{Pac. J. Math.}, \textbf{66} (1976), 111--118.

\bibitem[FW]{FW} \textsc{E. Fominykh, B. Wiest}, Upper bounds for the complexity of the torus knot complements, \textit{J. Knot Theory Ramifications}, \textbf{22} (2013), 1350053-1--1350053-19.

\bibitem[JRT]{JRT} \textsc{W. Jaco, H. Rubinstein, S. Tillmann}, Minimal triangulations for an infinite family of lens spaces, \textit{J. Topol.}, \textbf{2} (2009), 157--180.

\bibitem[JRT2]{JRT2} \textsc{W. Jaco, H. Rubinstein, S. Tillmann}, Coverings and Minimal Triangulations of 3-manifolds, \textit{Algebr. Geom. Topol.}, \textbf{11} (2011), 1257--1265.

\bibitem[MP]{MP} \textsc{B. Martelli, C. Petronio}, A new decomposition theorem for 3-manifolds, \textit{Ill. J. Math.}, \textbf{46} (2002), 755--780. 

\bibitem[MP2]{MP2} \textsc{B. Martelli, C. Petronio}, Complexity of Geometric Three-manifolds, \textit{Geom. Dedicata}, \textbf{108} (2004), 15--69. 

\bibitem[M1]{M1} \textsc{S.~Matveev},  Complexity theory of 3-dimensional manifolds, \textit{Acta Appl. Math.}, \textbf{19} (1990), 101--130.

\bibitem[M2]{M2} \textsc{S.~Matveev}, \textit{Algorithmic topology and classification of 3-manifolds}, ACM-Monographs, {\bf 9}, Spinger-Verlag, Berlin-Heidelberg-New York, 2003.

\bibitem[O]{O} \textsc{P. Orlik}, Seifert manifolds, Lecture Notes in Mathematics \textbf{291}, Spinger-Verlag, Berlin-Heidelberg-New York, (1972). 


\bibitem[Se]{Se} \textsc{H. Seifert}, Topologie dreidimensionaler gefaserter Ra$\ddot{\textup{u}}$me, 
\textit{Acta Math.}, \textbf{60} (1933), 147--238.

\bibitem[Sc]{Sc} \textsc{P. Scott},    The geometries of 3-manifolds, \textit{Bull. Lond. Math. Soc.}, \textbf{15} (1983), 401--487.



\bibitem[Wa]{Wa} \textsc{F. Waldhausen}, Eine Klasse von 3-dimensionalen Mannigfaltigkeiten. I, II, \textit{Invent. Math.} {\bf 3} (1967), 308--333; ibid. {\bf 4} (1967), 87--117. 


\end{thebibliography}
\end{document}